\documentclass[11pt]{amsart}

\title{Skew-adjoint maps and quadratic Lie algebras}
\author{Pilar Benito}
\email{pilar.benito@unirioja.es}
\author{Javier R\'andez-Ib\'a\~nez}
\email{jarandez@unirioja.es}
\author{Jorge Rold\'an-L\'opez}
\email{jorge.roldanl@unirioja.es}
\address{Departamento de Matem\'aticas y Computaci\'on, Universidad de La Rioja, Logro\~no, La Rioja, Spain}
\date{\today}

\makeatletter
\@namedef{subjclassname@2020}{%
  \textup{2020} Mathematics Subject Classification}
\makeatother


\usepackage[utf8]{inputenc}
\usepackage[T1]{fontenc}
\usepackage[english]{babel}

\usepackage{amsmath, amsthm, amssymb}
\usepackage{mathdots, enumerate, hyperref, breakcites}

\theoremstyle{plain}
	\newtheorem{lemma}{Lemma}[section]
	\newtheorem{proposition}[lemma]{Proposition}
	\newtheorem{theorem}[lemma]{Theorem}
	\newtheorem{corollary}[lemma]{Corollary}
\theoremstyle{definition}
	\newtheorem{definition}{Definition}[section]
\theoremstyle{remark}
	\newtheorem{remark}[lemma]{Remark}
	\newtheorem{example}{Example}[section]

\newcommand{\mfd}{\mathfrak{d}}

\newcommand{\mfh}{\mathfrak{h}}
\newcommand{\mfn}{\mathfrak{n}}
\newcommand{\mfsl}{\mathfrak{sl}}
\newcommand{\mfsu}{\mathfrak{su}}
\newcommand{\mbC}{\mathbb{C}}
\newcommand{\mbK}{\mathbb{K}}
\newcommand{\mbR}{\mathbb{R}}

\DeclareMathOperator{\ad}{ad}
\DeclareMathOperator{\id}{Id} 
\DeclareMathOperator{\spa}{span} 
\DeclareMathOperator{\im}{Im} 
\DeclareMathOperator{\inner}{Inner} 
\DeclareMathOperator{\Tr}{Tr}
\DeclareMathOperator{\der}{Der}

\begin{document}
\maketitle
\begin{abstract}
	The procedure of double extension of vector spaces endowed with non-degenerate bilinear forms allows us to introduce the class of generalized $\mbK$-oscillator algebras over any arbitrary field $\mbK$. Starting from basic structural properties of such algebras and the canonical forms of skew-adjoint endomorphisms, we will proceed to classify the subclass of quadratic nilpotent algebras and characterize those algebras in the class with quadratic dimension 2. This will enable us to recover the well-known classification of real oscillator algebras, also known as Lorentzian algebras, given by Alberto Medina in 1985.
\end{abstract}


\section{Introduction}

In linear algebra, considerable work has been done on isometric, self-adjoint, and skew-adjoint endomorphisms with respect to a bilinear form. Canonical matrices for operators of this type on a complex inner product vector space have been provided in \cite{Horn_Merino_1999}. Similar results for more general fields can be found in \cite{Sergeichuk_1988} and, more recently, in \cite{Caalim_Futorny_Sergeichuk_Tanaka_2020}. In this paper, we will focus on skew-adjoint endomorphisms and their significance as a fundamental tool in constructing quadratic Lie algebras.

A \emph{Lie algebra} $L$ is a vector space over a field equipped with a bilinear product $[x,y]$ that satisfies $[x,x]=0$ and the \emph{Jacobi identity}:
\begin{equation}\label{eq:jacobi}
	J(x,y,z)=[[x,y],z]+[[y,z],x]+[[z,x],y]=0.
\end{equation}
The algebra $L$ is said to be \emph{quadratic}if it is endowed with a non-degenerate symmetric bilinear form, $\varphi$, which is \emph{invariant} with respect to the Lie bracket:
\begin{equation}\label{eq:invariant}
	\varphi([x,y],z)+\varphi(y,[x,z])=0.
\end{equation}
If $(L,\varphi)$ is quadratic, the invariance of the equation\eqref{eq:invariant} is equivalent to the left multiplication operators, called \emph{adjoint} or \emph{inner derivations}   denoted by $\ad x$, being $\varphi$-skew-adjoint endomorphisms. In fact, the set of inner derivations of $L$, $\inner (L)$, is an ideal of the whole algebra of derivations of $L$, $\der (L)$. And the set of $\varphi$-skew-adjoint derivations, $\der_\varphi (L, \varphi)=\{d\in \der (L): \varphi(d(x),y)+\varphi(x,d(y))=0\}$, is a Lie subalgebra of $\der L$ that contains $\inner (L)$.

Semisimple Lie algebras under their Killing-Cartan form, which is defined as $\kappa(x,y)=\Tr(\ad x\ad y)$, are nice examples of quadratic algebras. On the opposite structural side, we find abelian Lie algebras; all of them are quadratic by using any non-degenerate symmetric form. The orthogonal sum (as ideals) of semisimples and abelian algebras allows us to assert that reductive Lie algebras are also quadratic. According to \cite{Medina_Revoy_1985}, any non-simple, non-abelian, and indecomposable quadratic Lie algebra (i.e. the algebra does not break as an orthogonal sum of two regular ideals) is a double extension either by a one-dimensional or by a simple Lie algebra. The main tool enabling this classical procedure is the existence of skew-adjoint derivations. We point out that the class of quadratic Lie algebras is quite large and contains reductive Lie algebras and also infinitely many non-semisimple examples. Most of the examples, structures, and constructions on quadratic algebras have been set over fields of characteristic zero (see \cite{Ovando_2016} for a survey-guide). In positive characteristic, they have not been as extensively studied.

Over the reals, the double extension of any Euclidean vector space by a skew-adjoint automorphism yields to the class of real oscillator algebras. They are quadratic and solvable Lie algebras of dimension $2n+2$ and a bilinear invariant form of Lorentzian type (inner product with metric signature $(2n+1,1)$). The real oscillator class was firstly introduced and easily described, thanks to the Spectral Theorem on real skew-adjoint operators, by ALberto Medina in \cite[Section 4]{Medina_1985}. The name oscillator comes from quantum mechanics because they describe a system of a harmonic oscillator $n$-dimensional Euclidean space. At the same time, J. Hilgert and K. H. Hofmann arrive at oscillator algebras in their characterization of Lorentzian cones in real Lie algebras \cite{Hilgert_Hofmann_1985}. In \cite[Definition II.3.6]{Hilgert_Hofmann_Lawson_1989}, the authors term them as (solvable) Lorentzian algebras. For $n\geq 2$, the Levi subalgebra of the algebra of skew-derivations of a $(2n+2)$-dimensional oscillator algebra is the special unitary real Lie algebra $\mfsu_n(\mbR)$ (see \cite[Theorem 3.2]{Benito_RoldanLopez_2023b}). So, oscillator algebras can be doubly extended to a countable series of non-semisimple and non-solvable quadratic Lie algebras. Moreover, the study of other non-associative structures on oscillator algebras provides information on connections and metrics on oscillator Lie groups \cite[Section 5]{Albuquerque_Barreiro_Benayadi_Boucetta_SanchezDelgado_2021}.

Throughout this paper, we will extend the notion of real oscillator algebras to any arbitrary field $\mbK$ of characteristic different from 2. Under the name of generalized $\mbK$-oscillator algebras, we encode the double extensions of any abelian quadratic Lie algebra by any skew-adjoint endomorphisms. In the particular case $\mbK=\mbR$, extensions through skew-adjoint automorphisms allow us to recover the class of real oscillator algebras \cite[Lemme 4.2]{Medina_1985}.

The paper is organized as follows. In Section~\ref{s:skew-maps}, we assemble some basic properties, orthogonal decompositions, and canonical forms of skew-adjoint endomorphisms. Definition~\ref{def:one-by-abelian-extension} in Section~\ref{s:generalOscillator} establishes the concept of generalized oscillator algebras over arbitrary fields of characteristic not $2$, and Lemma~\ref{lem:general-oscillator-properties} reviews some of structural properties of this class of quadratic algebras. In Section~\ref{s:generalOscillator}, we also classify those algebras in the class that are nilpotent in Theorem~\ref{thm:oscillator-nilpotent}, and provide a characterization of those with quadratic dimension two. In Section~\ref{s:isolorentz}, we prove that indecomposable quadratic algebras with Witt index 1 are simple or solvable. The solvable ones are just the subclass of generalized oscillator algebras which are constructed as a double extension of skew-adjoint automorphisms. The proof of the assertion is based on the concept of \emph{isomaximal ideal} introduced in \cite{Kath_Olbrich_2004}. This result enables us to recover the well-established classification of real oscillator algebras given by Medina in 1985.


\section{Skew-maps on orthogonal subspaces}\label{s:skew-maps}

Let $V$ be a $\mbK$-vector space and $f\colon V\to V$ a $\mbK$-endomorphism. From now on, we denote $m_f(x)\in \mbK[x]$ as the (monic) minimal polynomial of $f$, that is $m_f(f)=0$ and any other polynomial $q(x)$ with $q(f)=0$ is a multiple of $m_f(x)$. The factorization of $m_f(x)$ into distinct irreducible polynomials, $m_f(x)=\pi_1^{k_1}(x)\cdots\pi_r^{k_r}(x)$, with $\pi_i(x)$ monic and $k_i\geq 1$, induces a direct sum vector decomposition of $V$, referred to as \emph{primary decomposition}:
\begin{equation}\label{eq:primary-decomposition}
	V=V_{\pi_1}\oplus \cdots\oplus V_{\pi_r}=V_0 \oplus\bigl( \underset{\pi_i\neq x}{\oplus}V_{\pi_i}\bigr),
\end{equation}
where $V_{\pi_i}=\{v\in V\mid \pi_i^{k_i}(f)(v)=0\}$. Each primary component is an $f$-invariant subspace. In the particular case $\pi_1(x)=x-\lambda$, the scalar $\lambda$ is an \emph{eigenvalue} of $f$, and $V_{x-\lambda}$ is usually denoted as $V_{\lambda}$. In this case, the subspace is called a \emph{generalized $\lambda$-eigenspace}. Thus, $V_0$ is the generalized $0$-eigenspace, and $V_0$ will be null if and only if $f$ is a bijective map. Let $J_n(\lambda)$ denote the Jordan $n$-by-$n$ canonical block and $C(\pi(x))$ the companion matrix of a given monic polynomial $\pi(x)=x^n-a_{n-1}x^{n-1}-\dots-a_1x-a_0$. Therefore,

\begin{align}\label{eq:jordan-canonical-blocks}
	J_n(\lambda):=\begin{pmatrix}
		              \lambda & 1       &        & 0       \\
		              0       & \lambda & \ddots & 0       \\
		                      &         & \ddots & 1       \\
		              0       & 0       & 0      & \lambda
	              \end{pmatrix},
	 &
	\, C(\pi(x)):=\begin{pmatrix}
		              0 &        & 0 & a_0     \\
		              1 & \ddots &   & a_1     \\
		                & \ddots & 0 & \vdots  \\
		              0 &        & 1 & a_{n-1}
	              \end{pmatrix}
\end{align}

Now, let $\varphi$ be a symmetric bilinear form, and assume that $f$ is $\varphi$-skew-adjoint (henceforth, we will use $\varphi$-skew to abbreviate this term), which means, $\varphi(f(x), y)=-\varphi(x, f(y))$.  Then, for any $s\geq 0$, we have:
\begin{equation}\label{eq:power-invarianza}
	\varphi(f^s(x),y)=(-1)^s\varphi(x,f^s(y))=\varphi(x,(-f)^s(y)).
\end{equation}
This implies that,
\begin{equation}\label{eq:skew-polinomial}
	\varphi(q(f)(x),y)=\varphi(x,q(-f)(y)) \text{\ for any\  } q(x)\in \mbK[x].
\end{equation}

From now on, the pair $(V,\varphi)$, where $\varphi$ is a symmetric bilinear form, will be called an orthogonal $\mbK$-vector space.

\begin{proposition}\label{prop:primary-decomposition-facts}
	Let $(V, \varphi)$ be an orthogonal  $\mbK$-vector space, and $f$ a $\varphi$-skew $\mbK$-endomorphism of $V$ with minimal polynomial $m_f(x)=\pi_1^{k_1}(x)\cdots \pi_r^{k_r}(x)$, and primary decomposition $V=\bigoplus_{i=1}^rV_{\pi_i}$. The orthogonal subspaces $V_{\pi_i}^\perp$ are $f$-invariant, and for any $\pi_i(x)$, we have two possibilities:
	\begin{enumerate}[\quad a)]
		\item $\pi_i(-x)\neq \pm \pi_j(x)$ for all $1\leq j\leq r$, in this case, $V_{\pi_i}\subset
			      V^\perp$ and $\varphi$ is a degenerate form.
		\item There is a unique  $j_i\in \{1, \dots, r\}$ such that
		      $\pi_i(-x)=(-1)^{\deg\pi_i}\pi_{j_i}(x)$ and then  $V_{\pi_i}^\perp=V_{\pi_{j_i}}\cap V_{\pi_i}^\perp\oplus\bigl(\underset{k\neq j_i}{\oplus}V_{\pi_k}\bigr)$. Moreover,

		      \begin{itemize}
			      \item for any $1\leq i\leq r$, $V_{\pi_{j_i}}\cap V_{\pi_i}^\perp=V_{\pi_{j_i}}\cap V^\perp$ and
			      \item if $i\neq j_i$, the primary components $V_{\pi_i}, V_{\pi_{j_i}}$ are totally isotropic.
		      \end{itemize}
	\end{enumerate}
\end{proposition}

\begin{proof}
	From equation~\eqref{eq:power-invarianza}, it is easily checked that $V_{\pi_i}^\perp$ is invariant under $f$. We point out that $ W=\underset{\pi_i\neq x}{\oplus}V_{\pi_i}\subseteq \im f$. Even more, $f^s\mid_{V_{\pi_i}}$ is bijective for any direct summand of $W$ with $s\geq 1$. Suppose $W\neq V$, which is equivalent to $f$ not being one-to-one. Then, the polynomial $x$ appears in the factorization of $m_f(x)$. Reordering, if necessary, we assume $x=\pi_1(x)=-\pi_1(-x)$. So, $V_{\pi_1}=V_x=V_0$. For any  $w\in V_{\pi_i}\subseteq W$, there exists a vector $v\in V_{\pi_i}$ such that $w=f^{k_1}(v)$. We have:
	\begin{equation*}
		\varphi(V_0, w)=\varphi(V_0, f^{k_1}(v))=(-1)^{k_1}\varphi(f^{k_1}(V_0),w)=0.
	\end{equation*}
	Thus, $V_0\perp W$, and any assertion regarding the primary component $V_{\pi_1}$ follows. This is a particular case of item b). Without loss of generality, we can assume $f$ is one-to-one. Since $\pi_i(x)$ is irreducible, so is $\pi_i(-x)$, thus either $\gcd(\pi_i(x),\pi_j(-x))=1$ or $\gcd(\pi_i(x),\pi_j(-x))=\pi_i(x)$. The second case happens if and only if $\pi_j(-x)=\pm\pi_i(x)$. Firstly, assume $\pi_i(-x)\neq \pm\pi_j(x)$ for any $1\leq j\leq r$. Then, $\gcd(\pi_i^{k_i}(-x), \pi_j^{k_j}(x))=1$, so we can use Bezout's identity: there exist $a(x)$, $b(x)\in\mbK[x]$ with $a(x)\pi_i^{k_i}(-x)+b(x)\pi_j^{k_j}(x)=1$. Take $v\in V_{\pi_i}$, so $\pi_i^{k_i}(v)=0$. Using Bezout's identity, for any vector $w$ in any primary component $V_{\pi_j}$, we have $w=a(f)\pi_i^{k_i}(-f)(w)$ and therefore,
	\[
		0=\varphi(a(-f)(\pi_i^{k_i}(f)(v)),w)=\varphi(v,a(f)\pi_i^{k_i}(-f)(w))=\varphi(v,w).
	\]
	This implies $0\neq V_{\pi_i}\subset V^\perp$, and a) follows. Otherwise, the uniqueness of the decomposition into irreducibles of the minimal polynomial ensures that there exists a unique $1\leq j_i\leq r$ such that $\pi_i(-x)=(-1)^{\deg\pi_i}\pi_{j_i}(x)$. This implies that $\pi_i(-x)$ and $\pi_k(x)$ are coprime for any $k\neq j_i$. Using  the Bezout's identity as in the previous reasoning, $V_{\pi_i}\perp V_{\pi_k}$, and the direct sum decomposition of $V_{\pi_i}^\perp$ follows from the $f$-invariance of  $V_{\pi_i}^\perp$. This decomposition also shows that $V_{\pi_i}\subseteq V_{\pi_i}^\perp$ when $i\neq j_i$. Finally, we will prove the equality $V_{\pi_{j_i}}\cap V_{\pi_i}^\perp=V_{\pi_{j_i}}\cap V^\perp$. From $V^\perp \subseteq V_{\pi_i}^\perp$ we have that $V_{\pi_{j_i}}\cap V^\perp\subseteq V_{\pi_{j_i}}\cap V_{\pi_i}^\perp$. The equality will follow by proving that $V_{\pi_{j_i}}\cap V_{\pi_i}^\perp\subseteq V^\perp$. Let $v\in V_{\pi_{j_i}}\cap V_{\pi_i}^\perp$, and note that $\varphi(v,w)=0$ if $w\in V_{\pi_i}$. If  $w\in V_{\pi_k}$ and $k\neq i$, since $V_{\pi_k}\perp V_{\pi_{j_i}}$, we also have $\varphi(v,w)=0$. Therefore, $v\in V^\perp$.

\end{proof}

\begin{remark}\label{rmk:polinomio-sim}
	For any polynomial $p(x)=x^n+a_{n-1}x^{n-1}+\dots+a_1x+a_0$ such that $a_0\neq 0$, $q(x)=x^n-a_{n-1}x^{n-1}+\dots+(-1)^{n-1}a_1x+(-1)^na_0$ is the unique monic polynomial fulfilling $p(-x)=(-1)^{\deg p}q(x)$. The set of roots of $q(x)$ is just $\{-\lambda: p(\lambda)=0\}$. In particular, $q(x)=p(x)$ if and only if the monomials of $p(x)$ have even degree, and the roots of $p(x)$ are of the form $\pm \lambda_1,\dots, \pm\lambda_n$.
\end{remark}

From the previous Proposition~\ref{prop:primary-decomposition-facts}, we arrive at the following general orthogonal decomposition of $(V, \varphi)$ through a $\varphi$-skew map.

\begin{corollary}\label{cor:descomposicion-ortogonal}
	Let $(V,\varphi)$ be an orthogonal $\mbK$-vector space over a field of characteristic not $2$, and $f\colon V\to V$ be a $\varphi$-skew linear map. Then, up to permutation, the factorization of $m_f(x)=x^\alpha\pi_1^{k_1}(x)\cdots \pi_r^{k_r}(x)$ into irreducible monic polynomials splits as $m_f(x)=x^\alpha p(x)q(x)s(x)n(x)$ where $\alpha\geq 0$, and if the degree of some of the $p,q,s,n$ factors is $\geq 1$:
	\begin{enumerate}[\quad a)]
		\item $p(x)=\pi_1^{k_1}(x)\cdots \pi_t^{k_t}(x)$ and $\pi_i(-x)=\pm \pi_i(x)$,
		\item $q(x) =\pi_{t+1}^{k_{t+1}}(x)\cdots \pi_{t+l}^{k_{t+l}}(x)$ and $\pi_j(-x)=\pm \pi_{l+j}(x)$ and,
		\item $s(x)=\pi_{t+l+1}^{k_{t+l+1}}(x)\cdots \pi_{t+2l}^{k_{t+2l}}(x)$ and $ \pi_{l+j}(-x)=\pm\pi_j(x)$,
		\item $n(x)=\pi_{t+2l+1}^{k_{t+2l+1}}(x)\cdots \pi_r^{k_r}(x)$ and $\pi_j(-x)\notin \pm\{ \pi_1, \dots,\pi_r\}$.
	\end{enumerate}
	This yields to the orthogonal sum, $V=V_0\perp V^1\perp V^2\perp V^3$, where $V^1= \underset{\pi_i\mid p(x)}{\oplus}V_{\pi_i}$, $V^2= \underset{\pi_j\mid q(x)}{\oplus}\bigl(V_{\pi_j}\oplus V_{\pi_{j+l}}\bigr)$ with $V_{\pi_j}$ and $V_{\pi_{j+l}}$ totally isotropic subspaces, and $V^3= \underset{\pi_i\mid n(x)}{\oplus}V_{\pi_i}\subseteq V^\perp$.
\end{corollary}

\begin{corollary}\label{cor:semisimplicidad}
	Let $(V,\varphi)$ be an orthogonal subspace over a field $\mbK$ of characteristic not $2$ such that $\varphi$ is non-degenerate, and denote by $I_\varphi=\{v\in V: \varphi(v,v)=0 \}$ the set of $\varphi$-isotropic vectors. Let $f\colon V\to V$ be a $\varphi$-skew linear map with minimal polynomial $m_f(x)=x^\alpha\pi_1^{k_1}(x)\cdots \pi_r^{k_r}(x)$. Then:
	\begin{enumerate}[\quad a)]
		\item If $I_\varphi={0}$, $f$ is semisimple and $V=V_0\oplus V^1=\ker f\perp \im f$, so $m_f(x)=x^\alpha \pi_1(x)\dots \pi_t(x)$, $\alpha=0,1$. Moreover, any irreducible $\pi_i(x)$ is of the form $\pi_i(x)=x^{2n_i}-a_{n_i-1}x^{2(n_i-1)}-\dots -a_1x^2-a_0$, $n_i\geq 1$, and $a_0\neq 0$. In particular, $f=0$ is the only possibility if the base field is algebraically closed.
		\item There exists a permutation $\sigma\in S_r$, $\sigma^2=1$, such that $\pi_i(-x)=(-1)^{\deg\pi_i}\pi_{\sigma(i)}(x)$. In particular, $m_f(-x)=(-1)^{\deg m_f}m_f(x)$, and either $m_f(x)=x^\alpha$ or $m_f(x)=x^\alpha p(x)$, $p(0)\neq 0$, and the monomials of $p(x)$ are of even degree. So, the nonzero roots of $m_f(x)$ are of the form $\pm \lambda_1,\dots, \pm\lambda_n$.
		\item $V=V_0\perp V^1\perp V^2$, $V_{\pi_{\sigma(i)}}\cap V_{\pi_i}^\perp=0$ and the components $V_{\pi_i}$ and $V_{\pi_{\sigma(i)}}$ are equidimensional. The direct sum decomposition in $V^1$ is orthogonal, and $\bigl(V_{\pi_i}\oplus V_{\pi_{\sigma(i)}}\bigr)\perp \bigl(V_{\pi_j}\oplus V_{\pi_{\sigma(j)}}\bigr)$ for $j\neq i, \sigma(i)$.
		\item If the base field $\mbK$ is algebraically closed, $V=V_0\perp V^2$, and $m_f(x)=x^{\alpha}(x-\lambda_1)^{k_1}(x+\lambda_1)^{k_1}\cdots (x-\lambda_r)^{k_r}(x+\lambda_r)^{k_r}$.
	\end{enumerate}
\end{corollary}

\begin{proof}
	Suppose that there are no nonzero isotropic vectors. Thus $V=V_0\oplus V^1$ follows from Corollary~\ref{cor:descomposicion-ortogonal}. If $\alpha\geq 2$,  there is a nonzero vector $v\in V_0$ such that $f^\alpha(v)=0\neq f^{\alpha-1}(v)$, and then $\varphi(f^{\alpha-1}(v), f^{\alpha-1}(v))=-\varphi(f^{\alpha-2}(v), f^{\alpha}(v))=0$. So $f^{\alpha-1}(v)=0$, a contradiction. Now consider any irreducible factor $\pi_i(x)\neq x$. From Remark~\ref{rmk:polinomio-sim}, $\pi_i(x)=x^{2n_i}-\sum_{k=0}^{n_i-1} a_{k}x^{2k}$ with $a_0\neq 0$. Denote by $ \textbf{E}_{1,2n_i}$ the $2n_i\times 2n_i$ elemental matrix ($a_{r,s}=0$ for $(r,s)\neq (1,2n_i)$ and $a_{1,2n_i}=1$), and let $C(\pi_i(x))$ be  the companion matrix described in equation~\eqref{eq:jordan-canonical-blocks}. If $k_i\geq 2$, there is a set $T$ of $4n_i$ linearly independent vectors such that $U=\spa\langle T\rangle$ is $f$-invariant, and the matrix of $f|_U$ is of the form
	\begin{equation}\label{eq:canonical-blocks}
		\left(\begin{array}{ c | c }
				                    &                            \\
				C(\pi_i(x))         & \bold{0}_{2n_i\times 2n_i} \\
				                    &                            \\
				\hline
				                    &                            \\
				\textbf{E}_{1,2n_i} & C(\pi_i(x))                \\
				                    &
			\end{array}\right)
	\end{equation}

	Then we can find  $v,w\in T$ such that $f^{2n_i}(v)=\sum_{k=0}^{n_i-1} a_{k} f^{2k}(v)+w$ and $f^{2n_i}(w)=\sum_{k=0}^{n_i-1} a_{k} f^{2k}(w)$. So $w=\frac{1}{a_0}(-\sum_{k=1}^{n_i-1}a_kf^{2k}(w)+f^{2n_i}(w))$.  Using previous equalities, we have that:
	\[
		\begin{split}
			\varphi(v,w) & =\frac{1}{a_0}(\sum_{k=1}^{n_i-1} -a_k\varphi(v,f^{2k}(w))+\varphi(v,f^{2n_i}(w)))                                  \\
			             & =\frac{1}{a_0}(\sum_{k=1}^{n_i-1} -a_k\varphi(f^{2k}(v),w)+\varphi(f^{2n_i}(v),w))                                  \\
			             & =\frac{1}{a_0}(\sum_{k=1}^{n_i-1} -a_k\varphi(f^{2k}(v),w)+\sum_{k=0}^{n_i-1} a_k\varphi(f^{2k}(v),w)+\varphi(w,w)) \\
			             & =\frac{1}{a_0}(a_0\varphi(v,w)+\varphi(w,w))=\varphi(v,w)+\frac{1}{a_0}\varphi(w,w).
		\end{split}
	\]
	And thus, $\varphi(w,w)=0$, a contradiction. This proves assertion a). From $V^\perp=0$, the decomposition in c) and  items b) and d) follow from Proposition~\ref{prop:primary-decomposition-facts} and Remark~\ref{rmk:polinomio-sim}. Finally, to get $\dim V_{\pi_{\sigma(i)}}=\dim V_{\pi_i}$, note that $V_{\pi_{\sigma(i)}}\cap V_{\pi_i}^\perp=V_{\pi_{i}}\cap V_{\pi_{\sigma(i)}}^\perp=0$ because of $V^\perp=0$ and:
	\[
		\dim V_{\pi_{\sigma(i)}}^\perp+\dim V_{\pi_{\sigma(i)}}=\dim V=\dim V-\dim V_{\pi_i}+\dim V_{\pi_{\sigma(i)}},
	\]
	From previous equalities, we conclude the equidimensionality.
\end{proof}

\begin{example}\label{ex:teorema-espectral-real}
	Applying previous Corollary~\ref{cor:semisimplicidad} we get the well-know Spectral Theorem (skew-Hermitian case). Let $(V, \varphi)$ be a real orthogonal vector space without isotropic vectors, and $f\colon V\to V$ a nonzero $\varphi$-skew linear map. Then either $\varphi(v,v)>0$ for all $v\in V$ (Euclidean case) or $\varphi(v,v)<0$ for all $v\in V$. Moreover, $V=\ker f\perp \im f$, and there are $0<\lambda_1\leq\lambda_2\leq\cdots\leq\lambda_t$ such that $m_f(x)=x(x^2+\lambda_1)\cdots(x^2+\lambda_t)$. So, the eigenvalues of $f$ are either $0$ or purely imaginary  ($\pm i\sqrt{\lambda_k}, i^2=-1, i\in \mbC$). In the Euclidean case, for any nonzero vector  $v\in V_{x^2+\lambda_i}$, $f^2(v)=-\lambda_i v$. Thus the vector space $0\neq U_i(v)=\spa\langle v, f(v)\rangle$ is $2$-dimensional and $f$-invariant. Moreover, $\varphi(v,f(v))=0$ and $\varphi(f(v),f(v))=\lambda_i\varphi(v,v))>0$. Thus  $U_i(v)\cap U_i(v)^\perp=0$, and therefore $V=U_i(v)\oplus U_i(v)^\perp$. Even more, assuming $\varphi(v,v)=1$, $\{v_{i,1}=v, v_{i,2}=-\frac{f(v)}{\sqrt{\lambda_i}}\}$ is an orthonormal basis of $U_i(v)$ such that $f(v_{i,1})=-\sqrt{\lambda_i}v_{i,2}$ and $f(v_{i,2})=\sqrt{\lambda_i}v_{i,1}$. Since $U_i(v)$ and $U_i(v)^\perp$ are $f$-invariant subspaces, just reducing dimension by orthogonal decompositions and rescaling, we get an orthonormal basis of $V$ such that the pair $(f, \varphi)$ is represented by the pair of matrices $(A_f, \id_{\dim V})$ and $A_f$ is the matrix that decomposes as diagonal blocks
	\begin{equation}\label{spectral-real-decomposition}
		A_f=\Big(\underbrace{0, \dots, 0}_{d_0}, \left(\begin{array}{ c c }
				0                 & \sqrt{\lambda_1} \\
				-\sqrt{\lambda_1} & 0
			\end{array}\right)^{d_1}, \dots, \left(\begin{array}{ c c }
				0                 & \sqrt{\lambda_t} \\
				-\sqrt{\lambda_t} & 0
			\end{array}\right)^{d_t}\Big).
	\end{equation}
	Here $d_i=\frac{1}{2}\dim V_{x^2+\lambda_i}$ is the number of $2\times 2$ matrix $\lambda_i$'s blocks and $d_0=\dim \ker f$ denotes the number of $1\times 1$ matrix $0$'s blocks. For $f|_{V_0}$, we use any orthonormalization process. This is just the Spectral Theorem for real matrices.

	For any arbitrary field, assuming no isotropic vectors and the irreducible semisimple decomposition $m_f(x)=x(x^2+\mu_1)\cdots(x^2+\mu_t)$, the pair $(f, \varphi)$ is represented by the pair of matrices $(A_f, B_\varphi)$ where $B_\varphi$ is a diagonal matrix and
	\begin{equation}	\label{spectral-general-noniso-decomposition}
		A_f=\Big(\underbrace{0, \dots, 0}_{d_0}, \left(\begin{array}{ c c }
				0 & - \mu_1 \\
				1 & 0
			\end{array}\right)^{d_1}, \dots, \left(\begin{array}{ c c }
				0 & - \mu_t \\
				1 & 0
			\end{array}\right)^{d_t}\Big).
	\end{equation}

\end{example}
Next, let us return to assertion c) in Corollary~\ref{cor:semisimplicidad}, and consider the orthogonal decomposition $V=V_0\perp V^1\perp V^2$. The subspace $V_0$ and the subspaces of $V^2$ of the form $V_\lambda \oplus V_{-\lambda}$ (associated with the primary components $V_{\pi(x)}$ where $\pi(x)=x\pm\lambda$) admit canonical forms determined by the Jordan blocks $J_n(\lambda)$ described in equation~\eqref{eq:jordan-canonical-blocks}.

\begin{theorem}\label{thm:generalized-vap-decomposition}
	Let $(V, \varphi)$ an orthogonal vector space over an arbitrary field of characteristic not two such that $\varphi$ is non-degenerate. Suppose that $f\colon V\to V$ is a $\varphi$-skew automorphism, and $\lambda$ is a nonzero $f$-eigenvalue. Then $V=(V_{\lambda}\oplus V_{-\lambda})\oplus (V_{\lambda}\oplus V_{-\lambda})^\perp$ and $(V_{\lambda}\oplus V_{-\lambda})^\perp$ is a $f$-invariant subspace. The primary components $V_{\lambda}$, $V_{-\lambda}$, are equidimensional and totally isotropic subspaces. Even more, $V_\lambda\oplus V_{-\lambda}$ decomposes as orthogonal sum of a finite number of $f$-invariant subspaces $(W_i, \varphi|_{W_i})$ that have a basis in which the pair $(f|_{W_i}, \varphi|_{W_i})$ is expressed by a matrix pair $(A_{f|_{W_i}},B_{\varphi|_{W_i}})$ of the following type:
	\begin{equation}\label{eq:isot-nonzero-vap}
		\Big( \left(\begin{array}{ c c }
					J_{n_i}(\lambda) & 0                   \\
					0                & -J_{n_i}(\lambda)^t
				\end{array}\right), \left(\begin{array}{ c c }
					0         & \id_{n_i} \\
					\id_{n_i} & 0
				\end{array}\right)\Big).
	\end{equation}

	And, for the zero eigenvalue, $V_0$ is an orthogonal sum of $f$-invariant subspaces $U_j$ that have a basis in which the pair $(f|_{U_j}, \varphi|_{U_j})$ is expressed by a matrix pair $(A_{f|_{U_j}},B_{\varphi|_{U_j}})$ as in equation~\eqref{eq:isot-nonzero-vap} with $\lambda=0$ and $n_j=2k_j$ or (here $\mathbf{e}_1=(1,0,\dots,0)\in \mbK^{n_j}$)

	\begin{equation}\label{eq:isot-zero-vap}
		\Big(\left(
		\begin{array}{c|c|c}

				J_{n_j}(0)   &                 &               \\\hline
				\mathbf{e}_1 & 0               &               \\\hline
				             & -\mathbf{e}_1^t & -J_{n_j}(0)^t \\
			\end{array}
		\right), (-1)^{n_j}\mu_j\left(
		\begin{array}{c|c|c}

				          &   & \id_{n_j} \\\hline
				          & 1 &           \\\hline
				\id_{n_j} &   &           \\
			\end{array}
		\right)\Big).
	\end{equation}

\end{theorem}

\begin{proof}
	From Corollaries~\ref{cor:descomposicion-ortogonal} and~\ref{cor:semisimplicidad}, we establish that $V_{\mu}\cap V_{-\mu}^\perp=0$ and $V_{\mu}\subseteq V_{\mu}^\perp$ for $\mu=\lambda, -\lambda$. This readily yields $(V_{\lambda}\oplus V_{-\lambda})^\perp\cap (V_{\lambda}\oplus V_{-\lambda})=0$, and consequently, $V=(V_{\lambda}\oplus V_{-\lambda})\oplus (V_{\lambda}\oplus V_{-\lambda})^\perp$ due to the non-degeneracy of $\varphi$. Additionally, $\dim V_\mu=\dim V_{-\mu}$, and the multiplicities $k_{\mu}$ of the eigenvalues $\mu=\pm\lambda$ in the minimum polynomial of $f$ are equal. Therefore, $V_{\lambda}=\ker (f-\lambda\id)^{k_{\lambda}}$ and $V_{-\lambda}=\ker (f+\lambda\id)^{k_{\lambda}}$. Setting $k:=k_\lambda-1$, we can take a nonzero $v\in V_\lambda$ such that $(f-\lambda\id)^k(v)\neq 0$. Since $(f-\lambda\id)^k(v)\in V_\lambda$ and $V_{\lambda}\cap V_{-\lambda}^\perp=0$, there exists $w\in V_{-\lambda}$ with $\varphi((f-\lambda\id)^k(v),w)=\alpha\neq 0$. For any $u\in V$, equation~\eqref{eq:skew-polinomial} implies
	\[
		\varphi((f-\lambda\id)^k(v),u)=(-1)^k\varphi(v,(f+\lambda\id)^k(u)),
	\]
	and then $\varphi(v,(f+\lambda\id)^k(w))=(-1)^k\alpha\neq 0$, so $(f+\lambda\id)^k(w)\neq 0$. Rescaling $w$ if necessary, we can suppose $\alpha=1$.

	Next, define $v:=v_0, w:=w_0$, $v_r:=(f-\lambda\id)^r(v_0)$ and $w_s:=(f+\lambda\id)^s(w_0)$ for $1\leq r,s\leq k$. We point out that $\varphi(v_r, w_s)=(-1)^{s}\varphi(v_{r+s}, w_0)$ by applying equation~\eqref{eq:skew-polinomial}. We now set a new $w'=w'_0$ obtained in the linear span of $\{w_0 = w, w_1, \dots, w_k\}$, i.e.,
	\begin{equation*}
		w' = \sum_{i=0}^{k} \alpha_i w_i.
	\end{equation*}
	Those $\alpha_i$ coefficients are obtained as solutions from the system of equations for $j=0, \dots, k$:
	\begin{equation*}
		\varphi(v_j, w') = \delta_{jk},
	\end{equation*}
	where $\delta_{jk}$ is the Kronecker delta. This is a triangular system whose diagonal coefficients are $\pm 1$. Therefore there is a unique solution $w'$, and, as $\alpha_0 = 1$, $w' \in \ker(f+\lambda\id)^{k+1}$ but $w' \notin \ker(f+\lambda\id)^{k}$. Even more, $\varphi(v_k, w'_0)=1$ and $\varphi(v_r, w'_0)=0$ otherwise. The subspace
	\begin{equation*}
		W_1=\spa\langle v_r, w'_s: r,s\geq 0\rangle,
	\end{equation*}
	where $w'_s:=(-1)^s(f+\lambda\id)^s(w'_0)$, is $f$-invariant. And the ordered set $\{v_k,\dots, v_0, w'_0, \dots w'_k\}$ is a basis in which $(f|_{W_1}, \varphi|_{W_1})$ is expressed by a matrix pair as described in expression~\eqref{eq:isot-nonzero-vap} with $n_1=k_\lambda$. Since $W_1$ is regular, $V_{\lambda}\oplus V_{-\lambda}=W_1\oplus W_1^\perp$. Note that $W_1^\perp$ is $f$-invariant and regular, so recursively we get the desired ortogonal decomposition for the nonzero eigenvalue $\lambda$.

	In the sequel we assume that $0\neq V_0$, equivalent to $0$ being an eigenvalue of $f$. Let $k_0$ be the multiplicity of this eigenvalue in the minimum polynomial, and $k:=k_0-1$. Hence, $V_0=\ker f^{k+1}$ and there exists a nonzero $v\in V_0$ such that $f^k(v)\neq 0$. We will consider two different cases depending on the parity of $k$ and, in both cases, we will start with the previous $v$. Let us first suppose that $k=2n-1$ is odd and $k_0=2n$. Then,
	\[
		\varphi(v,f^k(v))=(-1)^k\varphi(f^k(v),v)=-\varphi(v,f^k(v)).
	\]
	Since $2\neq 0$, $\varphi(v,f^k(v))=0$ and, we can find an element $w\in V_0$ with $\varphi(w,f^k(v))\neq 0$ by means of the non-degenerancy of $\varphi$ and the equality $\varphi(f^k(v),V)=\varphi(f^k(v),V_0)$. Moreover, from $\varphi(w,f^k(v))=(-1)^k\varphi(f^k(w),v)$, we get $f^k(w)\neq 0$. This way, we have $0\neq v,\, w\in V_0$ such that $f^k(v)\neq0\neq f^k(w)$ and $\varphi(v,f^k(w))\neq 0$. In this context, the previous procedure for the nonzero eigenvalue $\lambda$ remains valid, allowing us to identify an $f$-invariant and regular vector space $U_1$ with a basis in which $(f|_{U_1}, \varphi|_{U_1})$ is expressed by a matrix pair as described in equation~\eqref{eq:isot-nonzero-vap} with $n_1=k_0=2n$ and $\lambda=0$.

	The other case arises when $k=2n$ is even, so $k_0=2n+1$. We can assume that $\varphi(v,f^k(v))\neq 0$. Otherwise, $\varphi(v,f^k(v))=0$ and, since $\varphi$ is non-degenerate, we could find $w\in V_0$ such that $\varphi(w,f^k(v))\neq 0$. Consequently, $f^k(w)\neq 0$. If $\varphi(w,f^k(w))\neq 0$, we can take $w$ instead of $v$. If not, consider $0\neq v'=v+w\in V_0$. We have that $f^k(v')\neq 0$ since $\varphi(v,f^k(v'))=\varphi(v,f^k(w))\neq 0$.

	Moreover, using the previous assumptions and $k=2n$, we have
	\[
		\varphi(v',f^k(v'))=\varphi(v,f^k(w))+\varphi(w,f^k(v))=2\varphi(v,f^k(w))\neq 0.
	\]

	We can then replace $v$ with $v'$ that fulfills the required condition.
	Starting next from $v\in V_0$ such that $0\neq \mu=\varphi(v, f^k(v))$, we define recursively $w_0=v$ and for $1\leq j\leq n$:
	\[
		w_{j}=w_{j-1}-\frac{\varphi(w_{j-1},f^{2n-2j}(w_{j-1}))}{2\varphi(w_{j-1}, f^{2n}(w_{j-1}))}f^{2j}(w_{j-1})
	\]
	A straightforward computation shows that $w_n$ satisfies $\varphi(w_n, f^{2n}(w_n))=\mu\neq 0$, and $\varphi(w_n, f^t(w_n))=0$ for any $1\leq t<2n$ (for $t$ odd numbers, apply equation~\eqref{eq:power-invarianza}). Even more,
	\[\varphi(f^s(w_n),f^t(w_n ))=\left\{\begin{array}{ll}
			0         & \text{if } s+t\neq 2n \\
			(-1)^s\mu & \text{if } s+t=2n
		\end{array}\right.\]
	Then, the subspace
	\[
		U_1=\spa \langle w_n, f(w_n), \dots, f^{2n}(w_n)\rangle,
	\]
	is $f$-invariant and regular, and the ordered set $\{v_0, \dots, v_{2n}\}$ where $v_j:=f^j(w_n)$ is a basis in which $(f|_{U_1}, \varphi|_{U_1})$ is expressed by a matrix pair as described in equation~\eqref{eq:isot-zero-vap-Caalim} with $n_1=n$, so $k_0=2n_1+1$. By reescaling and reordering in the following way
	\[
		v_{n-1}, v_{n-2}, \dots, v_0, v_n, -v_{n+1}, v_{n+2}, \dots, (-1)^nv_{2n},
	\]
	we arrive at the matrix pair described in equation~\eqref{eq:isot-zero-vap}.

	Concluding the proof, note that $U_1$ is a regular subspace, thus $V_0=U_1\oplus (U_1)^\perp$. Since $(U_1)^\perp$ is $f$-invariant and regular, recursively accounting for the parity in each step yields the desired orthogonal decomposition.
\end{proof}

\begin{remark}\label{rmk:split-skew}
	Over the complex field, the previous result is established in \cite[Theorem~2]{Horn_Merino_1999} as a classical characterization of the Jordan canonical forms of complex orthogonal and skew-symmetric matrices. The analogous result over algebraically closed fields of characteristic not 2 appears in \cite[Corollary~1.2]{Caalim_Futorny_Sergeichuk_Tanaka_2020}. Here, we provide an alternative proof that highlights a recursive constructive method based on straightforward linear arguments. The canonical form $(A_{f|_{U_j}}, B_{\varphi|_{U_j}})$ proposed in \cite{Caalim_Futorny_Sergeichuk_Tanaka_2020} is:
	\begin{equation}\label{eq:isot-zero-vap-Caalim}
		\Big(J_{2n_j+1}(0), \mu_j\left(\begin{array}{ cccc }
				  &    &         & 1 \\
				  &    & \iddots &   \\
				  & -1 &         &   \\
				1 &    &         &\end{array}\right)\Big), \quad \mu_j \in \mbK^\times.
	\end{equation}
\end{remark}


\section{Generalized oscillator algebras}\label{s:generalOscillator}
Along this section, we assume the base field $\mbK$ is of characteristic not $2$ unless otherwise stated.

In any quadratic solvable and non-abelian $n$-dimensional Lie algebra $(L, \varphi)$, such that $Z(L)\subseteq L^2=Z(L)^\perp$ the centre is a totally isotropic ideal. According to \cite[Proposition 2.9]{Favre_Santharoubane_1987}, any solvable quadratic Lie algebra (characteristic zero) contains a central isotropic element $z$, and the algebra $L$ can be obtained as a double extension of the $(n-2)$-dimensional quadratic and solvable Lie algebra $\frac{(\mbK\cdot z)^\perp}{\mbK\cdot z}$. Applying this one-step process iteratively, we get the class of solvable quadratic algebras from the class of quadratic abelian Lie algebras. Let's start this section with this construction.

\begin{example}[One-dimensional-by-abelian double extension]\label{ex:one-dim-dob-ext}
	Let $(V,\varphi)$ be an orthogonal $\mbK$-vector space with $\varphi$ non-degenerate, and $\delta$ be any $\varphi$-skew map. Denote by $\delta^*$ the dual $1$-form of $\delta$, so $\delta^*\colon\mbK\cdot \delta\to \mbK$ and $\lambda\delta\mapsto \lambda$. On the vector space $\mfd(V, \varphi, \delta):=\mbK \cdot \delta\oplus V\oplus \mbK \cdot \delta^*$ we introduce the binary product
	\begin{equation}\label{eq:bracket-oscillator}
		[t\delta + x+ s\delta^*,t'\delta + y+ s'\delta^*] :=  t\delta(y)-t'\delta(x) + \varphi(\delta(x), y)\delta^*,
	\end{equation}
	for $t,t',s,s' \in \mbK$, $x, y\in V$. It is easily checked that this product is skew and satisfies the Jacobi identity, $\underset{\circlearrowright a,b,c}{\sum}[[a,b],c]=0$. Then, $(\mfd(V, \varphi, \delta), [a,b])$ is a Lie algebra, and the bilinear form
	\begin{equation}\label{eq:bilinear-invariant}
		\varphi_\delta(t\delta + x+ s\delta^*,t'\delta +y+ s'\delta^*)=ts'+st'+\varphi(x,y)
	\end{equation}
	is symmetric, invariant, non-degenerate, and it extends $\varphi$ by the hyperbolic space $\spa_{\mbK}\langle\delta, \delta^*\rangle$. The method of constructing this algebra is known as \emph{double extension}. Over fields of characteristic zero, this procedure was introduced simultaneously by several authors in the 80s (see \cite{Bordemann_1997} and references therein), but starting from any quadratic Lie algebra $(L, \varphi)$ and any $\delta\in \der_\varphi L$. In this case, the bracket $[x,y]_L$ must be added in the binary product given by equation~\eqref{eq:bracket-oscillator}. To obtain new indecomposable quadratic algebras, it is important to take a non-inner $\delta$ $\varphi$-skew derivation \cite[Proposition 5.1]{FigueroaOFarrill_Stanciu_1996}. This one-dimensional construction is also valid in characteristic other than $2$ (see \cite[Theorem 2.2]{Bordemann_1997} for a more detailed explanation).
\end{example}

\begin{example}[Real oscillator algebras]\label{ex:real-oscillator}
	Now fix $\mbK=\mbR$ and consider an Euclidean vector space $(V,\varphi)$. According to Example~\ref{ex:teorema-espectral-real}, any $\varphi$-skew automorphism $\delta\colon V\to V$ is completely determined by an ordered $n$-tuple of positive real entries, $\lambda=(\lambda_1, \dots, \lambda_n)$ with $\lambda_i\leq \lambda_{i+1}$, and $\dim V=2n$. Even more, there is an orthonormal basis $\{x_1, \dots x_n, y_1,\dots, y_n\}$ for which $\delta(x_i)=-\lambda_iy_i$ and $\delta(y_i)=\lambda_ix_i$. By using the double extension procedure  given in Example~\ref{ex:one-dim-dob-ext} we obtain the series of quadratic real Lie algebras:
	\begin{equation}\label{eq:oscillator}
		(\mfd_{2n+2}(\lambda),\varphi_\delta)=(\spa_{\mbR}\langle\delta, x_1, \dots x_n, y_1,\dots, y_n, \delta^*\rangle, \varphi_\delta).
	\end{equation}
	The Lie bracket and the invariant bilinear form $\varphi_\delta$ are given in equation~\eqref{eq:bracket-oscillator} and~\eqref{eq:bilinear-invariant}. Real oscillator algebras are quadratic and solvable. The Witt index of $\varphi_\delta$ is $1$, so $\varphi_\delta$ is a Lorentzian form. This characteristic leads to these algebras also being referred to in the literature as (real) Lorentzian algebras (see \cite[Definition II.3.16]{Hilgert_Hofmann_Lawson_1989}). The one-dimensional-by-abelian construction of this class of algebras appears in \cite[Proposition II.3.11]{Hilgert_Hofmann_Lawson_1989}. The algebras obtained through this procedure are called \emph{standard solvable Lorentzian algebras}  of dimension $2n+2$, and they are denoted as $A_{2n+2}$.
\end{example}

Previous examples serve as the introduction to the definition of oscillator algebras over arbitrary fields of characteristic not $2$. The definition first appears in \cite{Benito_RoldanLopez_2023b}.

\begin{definition}\label{def:one-by-abelian-extension}
	Let $(V, \varphi)$ be a $\mbK$-vector space of dimension greater than or equal to $2$, endowed with a symmetric and non-degenerate bilinear form $\varphi$. For any $\varphi$-skew map $\delta\colon V\to V$ and its dual $1$-form linear map $\delta ^*\colon\mbK\cdot \delta\to \mbK$, described as $\lambda\delta\mapsto \lambda$, the one-dimensional-by-abelian double extension Lie algebra $\mfd(V, \varphi, \delta):=\mbK \cdot \delta\oplus V\oplus \mbK \cdot \delta^*$ with product as in equation~\eqref{eq:bracket-oscillator} and bilinear form as in expression~\eqref{eq:bilinear-invariant} will be referred to in this paper as a \emph{generalized $\mbK$-oscillator algebra} and as a \emph{$\mbK$-oscillator algebra} in the particular case where $\delta$ is automorphism.
\end{definition}

In the sequel, we use the following terminology. A Lie algebra $L$ is \emph{reduced} if $Z(L)\subseteq L^2$ and \emph{local} if $L$ has only one maximal ideal \cite[Definition 3.1]{Bajo_Benayadi_2007}. The \emph{derived series of $L$} is recursively defining as $L^{(0)}=L$  and $L^{(k+1)}=[L^{(k)}, L^{(k)}]$. If $L^{(k)}=0$ for some $k\geq 1$, it is said that $L$ is \emph{solvable}.  The \emph{descending central series of $L$} is defined with terms $L^1=L$ and $L^{k+1}=[L^{k}, L]$ and the \emph{upper central series} is defined from $Z_1(L)=Z(L)=\{x\in L:[x,L]=0\}$ (\emph{centre} of $L$) and $Z_{k+1}(L)=\{x\in L: [x,L]\subseteq Z_k(L)\}$. The algebra $L$ is \emph{nilpotent} if there exists $k\geq 2$ such that $L^k =0$. The smallest $k$ such that $L^k\neq 0$ and $L^{k+1=0}$ is the \emph{nilpotency index }of $L$. A quadratic algebra is said to be \emph{decomposable} if it contains a proper ideal $I$ that is non-degenerate ($I\cap I^\perp=0$, the ideal is also called a \emph{regular subspace}), and \emph{indecomposable} otherwise. Equivalently, $L$ is decomposable if and only if $L=I\oplus I^\perp$. In the literature, the terms \emph{reducible} and \emph{irreducible} are also used as synonyms for decomposable and indecomposable.

\begin{remark}\label{rmk:decomp-indecomp}
	In general, a non-reduced quadratic Lie algebra splits as an orthogonal sum, as ideals, of an abelian quadratic algebra and another reduced quadratic algebra. In characteristic zero, this assertion is just given in \cite[Theorem 6.2]{Tsou_Walker_1957}. Its proof also works in characteristics other than $2$. Therefore, any non-reduced quadratic Lie algebra is decomposable.
\end{remark}

Consider now the class of Lie algebras $\mfh$ that satisfy:
\begin{equation}\label{eq:heisenberg}
	\mfh^2=\mbK \cdot z=Z(\mfh).
\end{equation}
Since, for any $x, y \in \mfh$, $[x,y]=\lambda_{x,y}z$ and $\lambda_{x,y}\in \mbK$, the structure constants $\lambda_{x,y}$ allow us  to define in $\mfh$ the skew-symmetric bilinear form $\varphi_{\mfh}(x,y):= \lambda_{x,y}$. Then,  if $V$ is a complement vector space of the centre, i.e., $\mfh=V\oplus \mbK \cdot z$, the vector space is regular. The non-degeneracy of the skew-symmetric form $\varphi_\mfh|_V$  implies that there is a basis of $V$, $\{v_1, \dots, v_n, w_1, \dots, w_n,\}$ such that $\varphi_{\mfh}(v_i,w_j)=\delta_{ij}$ and $\varphi_{\mfh}(v_i, v_j)=\varphi_{\mfh}(w_i, w_j)=0$. Therefore, the algebras that satisfy equation~\eqref{eq:heisenberg} have odd dimension and exhibit a basis $\mathcal{B}=\{x_1, \dots, x_n, y_1, \dots, y_n, z\}$ such that $[v_i,w_j]=\delta_{ij}z$ and all other products are zero. Thus, for any $n\geq 1$, there is only one algebra of dimension $2n+1$ that satisfies equation~\eqref{eq:heisenberg}. In characteristic zero, these algebras are known as Heisenberg algebras. Throughout the paper, we will refer to them as \emph{generalized $\mbK$-Heisenberg algebras}.

\begin{lemma}\label{lem:general-oscillator-properties} Let $A=\mfd(V, \varphi, \delta)$ be a generalized $\mbK$-oscillator Lie algebra. Then the following holds:
	\begin{enumerate}[\quad a)]
		\item For any $k\geq 1$, $A^{k+1}=\im \delta^k \oplus  \mbK\cdot \delta^*$ if $\delta^k\neq 0$ and, in case that $\delta^k=0$,  $A^{m}=0$ for any $m\geq k+1$.
		\item For any $k\geq 1$, $Z_{k}(A)=\ker \delta^{k} \oplus \mbK\cdot \delta^*$  if $\delta^k\neq 0$ and, in case that $\delta^k=0$, $Z_{k}(A)=A$ for any $m\geq k$.
		\item $(Z_k (A))^\perp=A^{k+1}$ for any $k\geq 1$.
		\item $A$ is nilpotent if and only if $\delta$ is a nilpotent map. And the nilpotency index is the degree of the minimum polynomial of $\delta$.
	\end{enumerate}
	In addition, if $\delta\neq 0$,
	\begin{enumerate}[\qquad a)]
		\setcounter{enumi}{4}
		\item $Z(A)\subseteq A^2$ if and only if $\ker \delta\subseteq \im \delta$.
		\item $A^{(2)}= \mbK\cdot \delta^*$ and $A^{(3)}=0$. So, $A$ is solvable and the quotient $\displaystyle\frac{A^2+Z(A)}{Z(A)}$ is abelian.
		\item If $\delta$ is an automorphism, the dimension of $V$ is even, and the derived algebra $A^2=V\oplus \mbK \cdot \delta^*$ is a generalized $\mbK$-Heisenberg algebra. In particular, $V$ has a basis $\{v_1, \dots v_n, w_1, \dots, w_n \}$ such that $\varphi(\delta(v_i), v_j)=\varphi(\delta(w_i), w_j)=0$ and $\varphi(\delta(v_i), w_j)=\delta_{ij}$.
	\end{enumerate}
\end{lemma}

\begin{proof}
	Along the proof $W^\perp=\{x\in A: \varphi_\delta(W,x)=0\}$, thus $V^\perp=\spa\langle \delta, \delta^*\rangle$ and $\dim A=\dim W+\dim W^\perp$.

	From equation~\eqref{eq:bracket-oscillator},  $t_0\delta+x_0+s_0\delta^*\in Z(A)$ is equivalently to:
	\begin{equation}
		t_0\delta(y)-t'\delta(x_0) + \varphi(\delta(x_0), y)\delta^*=0,
	\end{equation}
	for all $t'\in \mbK$ and $y\in V$. If $\delta=0$, $A=Z(A)$ and then  $A^2=0$. Otherwise, the centre turns out to be $Z_1(A)=Z(A)=\ker \delta\oplus \mbK \cdot\delta^*$ and $A^2=\im \delta \oplus\spa\langle \varphi(\delta(x),y)\delta^*:  x,y\in V\rangle=\im \delta\oplus \mbK\cdot \delta^*$, the last equality is a consequence of the non-degeneracy of $\varphi$. This proves a) and b) when $k=1$. Assume $\delta^{k-1}\neq 0$ and $A^k=\im \delta^{k-1}\oplus  \mbK\cdot \delta^*$. For the $(k+1)$-term of the descending series, we have:
	\begin{equation}
		A^{k+1}=[A^k,A]=\spa\langle \delta^k(x),\varphi(\delta^k(x),y) \delta^*: x,y\in V\rangle.
	\end{equation}
	If $\delta^{k}=0$, that is, $\delta$ is a nilpotent map, $A^{k+1}=0=A^m$ for $m\geq k+1$, and $A$ is a nilpotent algebra. Otherwise, $\delta^k\neq 0$ and we can take $x\in V$ such that $\delta^k(x)\neq 0$. By the non-degenerancy of $\varphi$, there exists $y\in V$ such that $\varphi(\delta^k(x), y)\neq 0$. Then $A^{k+1}=\im \delta^k \oplus \mbK \cdot \delta^*$. This  proves item a) and the assertion from right to left in item d). Suppose now that $A$ is nilpotent and $k$ is its nilpotency index, i.e., $A^k\neq 0$ and $A^{k+1}=0$. The assumption implies $\delta^{k}=0\neq \delta^{k-1}$, so $k$ is just the degree of the minimum polynomial $m_\delta(x)$ and d) follows.

	Next, for the upper central series, we assume $\delta^{k}\neq 0$ and $Z_k(A)=\ker \delta^{k}\oplus \mbK\cdot \delta^*$. We will check the elements $x\in Z_{k+1}(A)$, so $[A,x]\subseteq Z_k(A)$. Since $\delta^*\in Z_{k+1}(A)$, we can assume $x\in \mbK\cdot \delta \oplus V$. If $\delta^{k+1}=0$, we have $[V,\delta]\subseteq \delta(V)\subseteq \ker \delta^{k}\subseteq Z_k(A)$ and then $Z_{k+1}(A)=A$ because $[V,V]=\mbK\cdot \delta^*$. Let's suppose that $\delta^{k+1}\neq 0$ and $[A, t_0\delta+x_0]\in Z_{k}(A)$. Equivalently, the two following elements are in $Z_k(A)$:
	\begin{equation}
		[\delta, t_0\delta+x_0]=\delta(x_0) \text{\ and \ }[v,  t_0\delta+x_0]=t_0\delta(v)+\varphi(\delta(x_0), v)\delta^*.
	\end{equation}
	That is $\delta(x_0), t_0\delta(v)\in \ker \delta^{k}$ for all $v\in V$. Since $\delta^{k+1}\neq 0$, the only possibility is $t_0=0$ and $x_0\in \ker \delta^{k+1}$. Therefore, $Z_{k+1}(A)=\ker \delta^{k+1}\oplus \mbK\cdot \delta^*$ and b) is proven.

	The statement c) regarding the orthogonality of the terms of the descending and upper central series follows by using $\varphi(\delta^k(x), y)=(-1)^k\varphi(x,\delta^k(y))$, so we have
	\begin{equation}
		A^{k+1}=\mbK\cdot \delta^*\oplus \im \delta^k \subseteq (\ker \delta^k)^\perp \cap (\mbK\cdot \delta^*)^\perp,
	\end{equation}
	the dimension formulae  $\dim A=\dim \ker \delta^k +\dim\, (\ker \delta^k)^\perp$, $\dim V=\dim \ker \delta^k + \dim \im \delta^k$, and the description of the terms in both series given in items a) and b).

	Finally, if $\delta$ is a nonzero map,  $A^2=\im \delta \oplus \mbK \cdot \delta^*$ and $Z(A)=\ker \delta \oplus \mbK \cdot \delta^*$. Therefore $A^{(2)}=[A^2,A^2]=\mbK \cdot \delta^*\subseteq Z(A)$, and items e) and f) follow easily. In the particular case that $\im \delta=V$, we have $A^2=V\oplus\, \mbK \cdot \delta^*$, and looking for $x\in V$ such that $[y, x]=\varphi(\delta(y), x)\delta^*=0$ for all $y\in V$, we get $\varphi(V, x)=0$. This implies $x=0$ from the non-degenerancy of $\varphi$. Then $Z(A^2)=\mbK \cdot \delta^*=A^{(2)}$ and therefore $A^2$ is a generalized $\mbK$-Heisenberg algebra. So there is a standard basis $\{v_1, \dots, v_n, w_1, \dots w_n, \delta^*\}$ such that $V=\spa \langle v_1, \dots, v_n, w_1, \dots w_n\rangle $ and $[v_i, w_j]=\varphi(\delta(v_i), w_j)\delta^*=\delta_{ij}\delta^*$, all other products being zero. Then item g) follows.
\end{proof}

\begin{remark}
	Statement c) in Lemma~\ref{lem:general-oscillator-properties},  that is, the orthogonal terms, $(A_{i+1})^\perp$, of the descending central series give us the upper central terms $Z_i(A)$ and vice versa, is well known for quadratic algebras in characteristic zero \cite{Medina_Revoy_1985}. According to assertion e), generalized $\mbK$-oscillator algebras are reduced if and only if $\delta$ is an automorphism or the length $n$ of any Jordan block $J_n(0)$ of $\delta$ is greater than or equal to $2$.
\end{remark}

\begin{remark}
	Indecomposable real quadratic Lie algebras satisfying that their quotient Lie algebra $\displaystyle\frac{A^2}{Z(A)}$ is abelian have been treated in \cite{Kath_Olbrich_2004}. In this paper, the authors give the classification of real  Lie algebras of maximal isotropic centre of dimension less or equal to $2$. The method they use is the \emph{two-fold extension}. In the next section, we will relate $\mbK$-oscillator algebras and quadratic algebras with maximal isotropic centre of dimension~$1$.
\end{remark}

Our previous Lemma~\ref{lem:general-oscillator-properties} ensures the existence of quadratic nilpotent Lie algebras over fields of characteristic different from $2$, with arbitrary nilpotence index, using one-dimensional-by-abelian double extensions and skew-nilpotent maps. As a corollary of the results in Section~\ref{s:skew-maps}, we can give the complete description of this type of algebras.

\begin{theorem}\label{thm:oscillator-nilpotent}
	Over fields of characteristic not $2$, any nonzero nilpotent generalized $\mbK$-oscillator algebra, $(\mfd(V, \varphi, \delta), \varphi_\delta)$, of nilpotent index $k$ is given  through a $\varphi$-skew nilpotent map $\delta$ such that $m_\delta(x)=x^k$. Even more, the orthogonal space $(V,\varphi)$ decomposes as orthogonal sum of a finite number of $\delta$-invariant regular subspaces, $(W, \varphi|_{W})$ of dimension $2n+1$ or $4m$ with $n\leq \lfloor \frac{1}{2} (k-1)\rfloor$ and $m\leq \lfloor \frac{1}{4}k\rfloor$. And there is a basis of $W$ in which the pair $(\delta|_W, \varphi|_W)$ is expressed by the matrix pair $(A_{\delta|_W}, B_{\varphi|_W})$ as in equation~\eqref{eq:isot-zero-vap} if $W$ has odd dimension and as in equation~\eqref{eq:isot-nonzero-vap} with $\lambda=0$ otherwise.
\end{theorem}

\begin{proof}
	The result follows from Theorem~\ref{thm:generalized-vap-decomposition} and Lemma~\ref{lem:general-oscillator-properties}.
\end{proof}

\begin{example}\label{ex:free-nilpotent-quadratic}
	In characteristic zero, the smallest non-abelian and nilpotent indecomposable quadratic Lie algebras are the \emph{free nilpotent} on $2$-generators $\mfn_{2,3}$ which is $3$-step nilpotent (i.e., $\mfn_{2,3}^2\neq 0=\mfn_{2,3}^3$) and it has dimension $5$, and the $6$-dimensional free nilpotent $\mfn_{3,2}$ on $3$-generators which is $2$-step nilpotent. This couple of algebras are the unique free-nilpotent algebras that admit quadratic structure \cite[Theorem 3.8]{delBarco_Ovando_2012}. They can be obtained as one-dimensional-by-abelian double extensions of the following orthogonal vector spaces:
	\begin{itemize}
		\item $\mfn_{2,3}$ is the double extension $\mfd(\mbK^3, \varphi_1, \delta_1)$ where the orthogonal vector space $(\mbK^3, \varphi_1)$ and the $\varphi_1$-skew map $\delta_1$ are given by means of the canonical form matrix pair ($A_{\delta_1}, B_{\varphi_1}$) as in equation~\eqref{eq:isot-zero-vap-Caalim} with $n_1=1$ and $\mu_1=-1$.
		\item $\mfn_{3,2}$ is the double extension $\mfd(\mbK^4, \varphi_2, \delta_2)$ where the orthogonal vector space $(\mbK^4, \varphi_2)$ and the $\varphi_2$-skew map $\delta_2$ are given by means of the canonical form matrix pair $(A_{\delta_2}, B_{\varphi_2})$ as in equation~\eqref{eq:isot-nonzero-vap} with $\lambda=0$ and $n_1=2$.
	\end{itemize}
\end{example}

The algebras in Example~\ref{ex:free-nilpotent-quadratic} admit other non-degenerate and invariant bilinear forms. In the case of $n_{3,2}$ all of their quadratic structures are isometrically isomorphic. This is clear from Theorem~\ref{thm:generalized-vap-decomposition}. The same is true for $n_{2,3}$ if $\mbK$ is algebraically closed, but not so clear from the previous theorem. If $\mbK=\mbR$, we have two non-isometrically isomorphic quadratic structures: the one given through $\mu_1=-1$ and that corresponding to $\mu_1=1$. This comment leads us to the concept of quadratic dimension.

The \emph{quadratic dimension} \cite{Bajo_Benayadi_1997} of a Lie algebra $L$ is defined as $d_q(L)=\dim B^s_{inv}(L)$ where $B^s_{inv}(L)$ is the subspace of symmetric invariant bilinear forms of $L$. Note that $L$ is quadratic if and only if $d_q(L)\geq 1$. For the quadratic free-nilpotent algebras in Example~\ref{ex:free-nilpotent-quadratic}, $d_q(n_{2,3})=4$ and $d_q(n_{3,2})=7$. In characteristic zero, any quadratic Lie algebra such that $d_q(L)=1$ is simple \cite[Theorem 3.1]{Bajo_Benayadi_1997}, and the converse is also true over algebraically closed fields. The paper \cite{Bajo_Benayadi_2007} is devoted to the structure of quadratic Lie algebras with quadratic dimension $2$. According to \cite[Lemma 3.1]{Bajo_Benayadi_2007} (characteristic zero), any indecomposable Lie algebra  whose quadratic dimension is $2$ is a local Lie algebra.

Let $A=\mfd(V, \varphi, \delta)$ be a generalized $\mbK$-oscillator algebra. Apart from $\varphi_\delta$, the bilinear form $\varphi_{1,0}$,  defined as $\varphi_{1,0}(\delta, \delta)=1$ and $\varphi_{1,0}(V\oplus \mbK \cdot \delta^*, A)=0$, is invariant because of $A^2\subseteq V\oplus \mbK \cdot \delta^*\subseteq A^{\perp_{\varphi_{1,0}}}$. So $d_q(\mfd(V, \varphi, \delta))\geq 2$. Even more, for any $v\in V\backslash \im \delta$, we can split $V=\mbK \cdot v \oplus U$ and $A^2\subseteq \mbK \cdot \delta^*\oplus U$. The bilinear form $T_{v,U}(v,v)=1$ and $T_{v,U}(\mbK \cdot \delta^*\oplus U\oplus \mbK \cdot \delta, A)=0$ is invariant and $T_{v,U}\notin \spa \langle \varphi_{1,0}, \varphi_\delta \rangle$. And we can define a fourth invariant and linearly independent form: $T'_{v,U}(\delta,v)=1$, $\delta, v$ isotropic vectors and $T'_{v,U}(\mbK \cdot \delta^*\oplus U,A)=0$. Thus $d_q(A)\geq 4$ if $\delta(V)\neq V$. In this case, $\ker\delta\neq 0$ and, from b) in Lemma~\ref{lem:general-oscillator-properties}, $Z(A)=\ker \delta \oplus \mbK \cdot \delta^*$ if $\delta\neq 0$, and therefore, $\dim Z(A)\geq 2$. This is a particular case of a more general result.

\begin{lemma}[Tsou, Walker, 1957]\label{lem:dq-dimension}
	Let $(A,\varphi)$ be a quadratic Lie algebra and $r=\dim Z(A)$. Then, $d_q(A)\geq 1+\frac{1}{2}r(r+1)$.
\end{lemma}

\begin{proof}
	Let $W$ be a complement subspace  of the subspace $A^2$ in $A$ and $w_1, \dots, w_d$ a basis of W. Then $A=W\oplus A^2$ and $d=\dim A-\dim A^2=r=\dim Z(A)$ because $(A^2)^\perp=Z(A)$. For each pair $(i,j)$, define the symmetric bilinear form on $W$  as $T_{i,j}(w_i, w_j)=1=T_{i,j}(w_j, w_i)$ and $T_{i,j}(w_k, w_s)=0$ for $(k,s)\neq (i,j)$. We extend $T_{i,j}$ to a symmetric form in $A$ by defining $T_{i,j}(A^2, A)=0$. We have that  $T_{i,j}$ is invariant because $A^2\subset A^{\perp_{T_{i,j}}}$. So, the vector space $\spa \langle T_{i,j}, \varphi: 1\leq i\leq j\leq d\rangle\subseteq B_{inv}^s(A)$ and, as the generator forms are linearly independent, the result follows.
\end{proof}

We recall that a Lie algebra is local if it has a unique maximal ideal. We note that, in the case of quadratic Lie algebras, $I$ is a maximal ideal if and only if $I^\perp$ is minimal. Therefore, having a unique maximal ideal is equivalent to having a unique minimal ideal in the class of quadratic algebras.

\begin{lemma}\label{lem:local-general-oscillator} Let $A=\mfd(V, \varphi, \delta)$ be a generalized $\mbK$-oscillator Lie algebra. The following assertions are equivalent:
	\begin{enumerate}[a)]
		\item $A$ is local.
		\item $A=A^2\oplus \mbK \cdot \delta$ and $\delta\neq 0$.
		\item $Z(A)$ is one dimensional.
		\item $\delta$ is an automorphism.
		\item The quadratic dimension of $A$ is 2.
	\end{enumerate}
	In this case, $A^2$ is a generalized $\mbK$-Heisenberg algebra. And the set of invariant symmetric forms is the vector space $B^s_{inv}(A)=\spa\langle \varphi_{1,0}, \varphi_\delta\rangle$ where $\varphi_{1,0}(A^2, A)=0$ and $\varphi_{1,0}(d,d)=1$.
\end{lemma}

\begin{proof}
	Any subspace $U$ containing $A^2$ is an ideal and any subspace of $Z(A)$ is an ideal. Assume firstly $A$ is local. As $\dim A\geq 4$,  $A$ is not abelian, so $\delta\neq 0$ and $Z(A)=\ker \delta \oplus \mbK \cdot \delta^*$. Since there is only one minimal ideal,  $\ker \delta=0$, so $\im \delta= V$, $A^2= V\oplus \mbK \cdot \delta^*$ and b) follows. From b) and $A^2=(Z(A))^\perp$, we have  $\dim A=\dim A^2+ \dim Z(A)$ and we get c). If $Z(A)=\mbK \cdot \delta^*\neq A$, using Lemma~\ref{lem:general-oscillator-properties} we have $\delta \neq 0$ and $Z(A)=\ker \delta\oplus \mbK \delta^*$, so $\ker \delta=0$ and $\delta$ is an automorphism. Then c) implies d). Now we will prove the final comment and $d)\Rightarrow e)$. From Lemma~\ref{lem:dq-dimension}, $d_q(A) \geq 2$. In fact, $\spa \langle\varphi_\delta,\varphi_{1,0}\rangle\subseteq B^s_{inv}(A)$. Let $\psi \in B^s_{inv}(A)$ and set $\alpha:=\psi(\delta, \delta)$  and $\beta:=\psi(\delta, \delta^*)$. Note that $\psi(\delta, x)=\psi(\delta^*, x)=0$ thanks to the invariance of $\psi$ and $x=[\delta, \delta^{-1}(x)]$. Consider $x_0,y_0\in V$ such that $0\neq k_0=\varphi (x_0, y_0)=\varphi(\delta(\delta^{-1}(x_0)), y_0)$. Then
	\[
		0=\psi([\delta^{-1}x_0,y_0], \delta^*)=\psi(k_0\delta^*,\delta^*)=k_0\psi(\delta^*, \delta^*).
	\]
	So, $\psi(\delta^*, \delta^*)=0$. Now, for any $x,y \in A$ and $y\neq 0$,
	\[
		\psi (x,y)=\psi(x, [\delta,\delta^{-1}(y)])=\psi([\delta^{-1}(y), x],\delta)=\psi(\varphi(y,x)\delta^*, \delta).
	\]
	Then, $\varphi(x,y)=\varphi(y,x)\psi(\delta^*,\delta)=\beta\varphi_\delta(x,y)$ and the equality also holds for $y=0$. In this way, we have that $\psi=\alpha \varphi_{1,0}+\beta\varphi_\delta$. Therefore, $d_q(A)=2$  and assertion e) follows. Assume finally e). From Lemma~\ref{lem:dq-dimension}, $r=\dim Z(A)=1$ so $Z(A)$ is a minimal ideal and therefore $\delta$ is an automorphism and $A^2$ is a maximal ideal. Let $I$ be a minimal ideal different from $Z(A)$, then $I\cap Z(A)=0$ and $A=I^\perp +A^2$. As $I^\perp$ is an ideal, applying the non-degenerancy of $\varphi_\delta$, we have $[I,I^\perp]=0$ and therefore $0\neq[A,I]=[A^2,I]\subseteq A^2\cap I \subseteq I$. The minimality of $I$ implies $I= A^2\cap I \subseteq A^2$ and there is a $0\neq v\in V$ such that $v\oplus t\delta^*\in I$, $t\in \mbK$. For any $w\in V$, we have $[v+ t\delta^*, w]=\varphi(\delta(v), w)\delta^* \in I$. Our assumption implies $\varphi(\delta(v), w)=0$ for all $w\in V$, a contradiction because $\varphi$ is non-degenerate and $\delta$ is an automorphism. This proves that $Z(A)$ is the unique minimal ideal and therefore $A$ is a local Lie algebra.
\end{proof}

\begin{remark}
	In characteristic zero, \cite[Theorem 3.1]{Bajo_Benayadi_2007} offers a characterization of local algebras that includes the class of $\mathbb{K}$-oscillators. The main goal in this paper is to provide examples and characterizations of algebras whose quadratic dimension is $2$. In fact, our proof of $d)\Rightarrow e)$  is the one presented in Proposition 4.1. This proposition asserts that the result is true for a double extension of any quadratic Lie algebra by any skew-symmetric derivation.
\end{remark}

In the sequel, we will tackle the case of isomorphisms and isometric isomorphisms in the subclass of $\mbK$-oscillator algebras. A similar result appears in \cite[Proposition 2.11]{Favre_Santharoubane_1987}. According to Definition~\ref{def:one-by-abelian-extension}, $\mbK$-oscillator algebras are one-dimensional-by-abelian double extensions of orthogonal subspaces through skew-automorphisms.

\begin{theorem}\label{thm:isomorphism-isometric-iso}
	Let $A_i=\mfd(V_i, \varphi_i,\delta_i)$  be two $\mbK$-oscillator algebras. Then, $A_1$ and $A_2$ are isomorphic if and only if there exists an isomorphism $f\colon V_1\to V_2$ and scalars $\lambda,\,\mu\in\mbK$ with $\lambda\mu\neq 0$ such that:
	\begin{enumerate}[\quad a)]
		\item $\delta_1=\mu f^{-1}\delta_2 f$.
		\item $\lambda\mu\varphi_1=\varphi_2(f(\cdot), f(\cdot))$.
	\end{enumerate}
	Moreover, they are isometrically isomorphic if and only if the previous conditions a) and b) stand with $\lambda\mu=1$.

	In addition, any isomorphism $F\colon A_1 \to A_2$ is completely determined the by the $5$-tuple $(f,z,\lambda, \mu, \nu)$ where $f\colon V_1\to V_2$ satisfies a) and b), $\lambda, \mu \in \mbK^\times$, $z\in V_1$, and $\nu \in \mbK$ in the following way:
	\[\begin{split}
			F(\delta_1) & =\mu\delta_2+z+\nu\delta_2^*, \\ F(x)&=f(x)+\varphi_2(\delta_2(z),f(\delta_1^{-1}(x)))\delta_2^*,\, x\in V_{2n},\\ F(\delta_1^*)&=\lambda\delta_2^*.
		\end{split}\]
	And $F$ will be an isometry if and only if $\lambda\mu=1$,
	\begin{equation}\label{eq:F-isometry}
		\mu\varphi_2(\delta_2(z),f(\delta_1^{-1}(x)))+\varphi_2(z,f(x))=0\text{\ and\ \,}2\mu\nu+\varphi_2(z,z)=0.
	\end{equation}
	(The last conditions are fulfilled by, for instance, setting $\nu=0, z=0$.)
\end{theorem}
\begin{proof} Assume firstly $A_1$ and $A_2$ are isomorphic. The centre of both algebras is one-dimensional and any isomorphism $F\colon A_1\to A_2$ induces an isomorphism of the centres, $F\mid_{Z(A_1)}\colon Z(A_1)\to Z(A_2)$. Therefore, there exists $\lambda\neq 0$ such that $F(\delta_1^*)=\lambda\delta_2^*$. For the derived algebras $A_i^2$, we also have the restricted isomorphism $F\mid_{A_1^2}\colon A_1^2\to A_2^2$, that acts as $F(x)=f(x)+g(x)\delta_2^*$ over the elements $x\in V_1$, where $f\colon V_{1}\to V_{2}$ and $g\colon V_{1}\to\mbK$ are linear maps. Furthermore, $f$ must be an isomorphism since it is surjective: take $y\in V_2$, there exists $x+k\delta_1^*\in A_1^2$ such that
	\[
		y=F(x+k\delta_1^*)=f(x)+(g(x)+k\lambda)\delta_2^*,
	\]
	impliying $f(x)=y$. Finally, we have that $F(\delta_1)=\mu\delta_2+z+\nu\delta_2^*$ with $\mu\neq 0$, because otherwise, we would have $F(A_1)\subset A_2^2\neq A_2$.

	With this, we conclude that there are scalars $\lambda$, $\mu$, $\nu\in\mbK$ with $\lambda\mu\neq 0$, an isomorphism $f\colon V_1\to V_1$, a linear map $g\colon V_1\to\mbK$ and $z\in V_2$ such that, for any $x\in V_1,$
	\[F(\delta_1)=\mu\delta_2+z+\nu\delta_2^*;\quad F(x)=f(x)+g(x)\delta_2^*;\quad F(\delta_1^*)=\lambda\delta_2^*.\]
	Since $F$ is an isomorphism, $F([\delta_1,x]_1)=[F(\delta_1),F(x)]_2$ for $x\in V_1$, so
	\[\begin{split}
			F(\delta_1(x))                          & =[\mu\delta_2+z+\nu\delta_2^*,f(x)+g(x)\delta_2^*]_2      \\
			f(\delta_1(x))+g(\delta_1(x))\delta_2^* & =\mu\delta_2(f(x))+\varphi_2(\delta_2(z),f(x))\delta_2^*.
		\end{split}\]
	Thus, for any $x\in V_1$, we have that $(f\circ\delta_1)(x)=(\mu\delta_2\circ f)(x)$, so condition a) follows, and $g(x)=\varphi_2(\delta_2(z),f(\delta_1^{-1}(x)))$. Taking $a\in V_1$ such that $\delta_1(a)=x$ and using again that $F$ is an isomorphism, we also have that
	\[\begin{split}
			F([a,y]_1)                            & =[F(a),F(y)]_2                                \\
			F(\varphi_1(\delta_1(a),y)\delta_1^*) & =[f(a)+g(a)\delta_2^*,f(y)+g(y)\delta_2^*]_2  \\
			\lambda\varphi_1(x,y)\delta_2^*       & =\varphi_2(\delta_2(f(a)),f(y))\delta_2^*     \\
			\lambda\varphi_1(x,y)\delta_2^*       & =\varphi_2(f(\delta_1(a))/\mu,f(y))\delta_2^* \\
			\lambda\varphi_1(x,y)\delta_2^*       & =\frac{1}{\mu}\varphi_2(f(x),f(y))\delta_2^*
		\end{split}\]
	Therefore, $\lambda\mu\varphi_1(x,y)=\varphi_2(f(x),f(y))$ for $x,y\in V_{2n}$, proving condition b). The isomorphism is defined by
	\[\begin{split}
			F(\delta_1) & =\mu\delta_2+z+\nu\delta_2^*, \\ F(x)&=f(x)+\varphi_2(\delta_2(z),f(\delta_1^{-1}(x)))\delta_2^*\; x\in V_{2n},\\ F(\delta_1^*)&=\lambda\delta_2^*.
		\end{split}\]

	If we additionally suppose that $F$ is an isometry, we have that
	\[
		\begin{split}
			\varphi_{1\delta_1}(\delta_1,\delta_1^*) & =\varphi_{2\delta_2}(F(\delta_1),F(\delta_1^*))                                                           \\
			1                                        & =\varphi_{2\delta_2}(\mu\delta_2+z+\nu\delta_2^*,\lambda\delta_2^*)                                       \\
			1                                        & =\lambda\mu,                                                                                              \\
			\varphi_{1\delta_1}(\delta_1,x)          & =\varphi_{2\delta_2}(F(\delta_1),F(x))                                                                    \\
			0                                        & =\varphi_{2\delta_2}(\mu\delta_2+z+\nu\delta_2^*,f(x)+\varphi_2(\delta_2(z),f(\delta^{-1}(x)))\delta^*_2) \\
			0                                        & =\mu\varphi_2(\delta_2(z),f(\delta_1^{-1}(x)))+\varphi_2(z,f(x)),                                         \\
			\varphi_{1\delta_1}(\delta_1,\delta_1)   & =\varphi_{2\delta_2}(F(\delta_1),F(\delta_1^*))                                                           \\
			0                                        & =\varphi_{2\delta_2}(\mu\delta_2+z+\nu\delta_2^*,\mu\delta_2+z+\nu\delta_2^*)                             \\
			0                                        & =2\mu\nu+\varphi_2(z,z).
		\end{split}
	\]
	resulting $\lambda\mu=1$ and satisfying the two conditions in equation~\eqref{eq:F-isometry}.

	For the converse implication, assuming that a) and b) hold, take any $z\in V_2$ and $\nu\in\mbK$, and define $F\colon A_1\to A_2$ by
	\[\begin{split}
			F(\delta_1) & =\mu\delta_2+z+\nu\delta_2^*,                                                    \\
			F(x)        & =f(x)+\varphi_2(\delta_2(z),f(\delta_1^{-1}(x)))\delta_2^*\text{\ if\ }x\in V_1, \\ F(\delta_1^*)&=\lambda\delta_2^*,
		\end{split}\]
	and extend it by linearity. This way, the map is linear and bijective. To prove that $F$ is an isomorphism, we just have to check $F([a,b]_1)=[F(a), F(b)]_2$ for the brackets $[\delta_1,\delta_1^*]_1$, $[\delta_1, x]_1$, $[x,y]_1$, $[x,\delta_1^*]_1$, $x,y\in V_1$.
	\[\begin{split}
			F([c,\delta_1^*]_1) & =F(0)=0=[F(a),\lambda\delta_2^*]_2=[F(a),F(\delta_1^*)]_2,\text{\ for any \ }c\in A_1                                                  \\
			F([\delta_1,x]_1)   & =F(\delta_1(x))=f(\delta_1(x))+\varphi_2(\delta_2(z),f(x))\delta_2^*                                                                   \\
			                    & =\mu\delta_2(f(x))+\varphi_2(\delta_2(z),f(x))\delta_2^*                                                                               \\
			                    & =[\mu\delta_2+z+\nu\delta_2^*,f(x)+\varphi_2(\delta_2(z),f(\delta_1^{-1}(x)))\delta_2^*]_2                                             \\
			                    & =[F(\delta_1),F(x)]_2,                                                                                                                 \\
			F([x,y]_1)          & =F(\varphi_1(\delta_1(x),y)\delta_1^*)=\lambda\varphi_1(\delta_1(x),y)\delta_2^*=\frac{1}{\mu}\varphi_2(f(\delta_1(x)),f(y))\delta_2^* \\
			                    & =\varphi_2(\delta_2(f(x)),f(y))\delta_2^*                                                                                              \\
			                    & =[f(x)+\varphi_2(\delta_2(z),f(\delta_1^{-1}(x)))\delta_2^*,f(y)+\varphi_2(\delta_2(z),f(\delta_1^{-1}(y)))\delta_2^*]_2               \\
			                    & =[F(x),F(y)]_2.
		\end{split}\]
	If we also assume $\lambda\mu=1$ and
	\[\mu\varphi_2(\delta_2(z),f(\delta_1^{-1}(x)))+\varphi_2(z,f(x))=0,\quad 2\mu\nu+\varphi_2(z,z)=0,\]
	then we have that:
	\[\begin{split}
			\varphi_{1\delta_1}(\delta_1,\delta_1^*)=1    & =\lambda\mu=\varphi_{2\delta_2}(\mu\delta_2+z+\nu\delta_2^*,\lambda\delta_2^*)=\varphi_{2\delta_2}(F(\delta_1),F(\delta_1^*)), \\
			\varphi_{1\delta_1}(x+k\delta_1^*,\delta_1^*) & =0=\varphi_{2\delta_2}(f(x)+\varphi_2(\delta_2(z),f(\delta_1^{-1}(x)))\delta_2^*+k\lambda\delta_2^*,\lambda\delta_2^*)         \\
			                                              & =\varphi_{2\delta_2}(F(x+k\delta_1^*),F(\delta_1^*)),                                                                          \\
			\varphi_{1\delta_1}(x,y)                      & =\varphi_1(x,y)=\frac{1}{\lambda\mu}\varphi_2(f(x),f(y))                                                                       \\
			=\varphi_{2\delta_2}(f(x)                     & +\varphi_2(\delta_2(z),f(\delta_1^{-1}(x)))\delta_2^*,f(y)+\varphi_2(\delta_2(z),f(\delta_1^{-1}(y)))\delta_2^*)               \\
			                                              & =\varphi_{2\delta_2}(F(x),F(y)),                                                                                               \\
			\varphi_{1\delta_1}(\delta_1,x)               & =0=\mu\varphi_2(\delta_2(z),f(\delta_1^{-1}(x)))+\varphi_2(z,f(x))                                                             \\
			                                              & =\varphi_{2\delta_2}(\mu\delta_2+z+\nu\delta_2^*,f(x)+\varphi_2(\delta_2(z),f(\delta_1^{-1}(x)))\delta_2^*)                    \\
			                                              & =\varphi_{2\delta_2}(F(\delta_1),F(x)),                                                                                        \\
			\varphi_{1\delta_1}(\delta_1,\delta_1)        & =0=2\mu\nu+\varphi_2(z,z)=\varphi_{2\delta_2}(\mu\delta_2+z+\nu\delta_2^*,\mu\delta_2+z+\nu\delta_2^*)                         \\
			                                              & =\varphi_{2\delta_2}(F(\delta_1),F(\delta_1)).
		\end{split}\]
	This proves that, in this case, $F$ is also an isometry.
\end{proof}


\section{Isomaximality and Lorentzian algebras}\label{s:isolorentz}

Along this section $\mbK$ is a field of characteristic zero.

According to \cite[Definition II.3.16]{Hilgert_Hofmann_Lawson_1989}, a Lorentzian Lie algebra is a pair $(L, \varphi)$ with a real Lie algebra $L$ and $\varphi$ an invariant and non-degenerate Lorentzian form, i.e., a real invariant symmetric bilinear form with signature $(p,1)$ where $p$ is the number of positive eigenvalues and $q$ is the number of negative eigenvalues. Section 6 of Chapter II in \cite{Hilgert_Hofmann_Lawson_1989} focuses on these algebras and explores the study of Lie semialgebras in them. The section points out that the complete classification of Lorentzian algebras is reduced to the indecomposable (named as irreducible by the authors) subclass, see \cite[Remark II.6.1]{Hilgert_Hofmann_Lawson_1989}. And the classification of this subclass is fully covered by Theorem II.6.14. Throughout this final section, we will prove that, over any arbitrary field $\mbK$ of characteristic zero, the class of solvable quadratic irreducible Lie algebras with quadratic form of Witt index $1$ (only one hyperbolic plane) is just the class of one-dimensional-by-abelian double extension through skew-automorphisms of orthogonal subspaces without isotropic vectors (i.e., the class of $\mbK$-oscillator algebras according to Definition~\ref{def:one-by-abelian-extension}). The Witt index $1$ condition generalizes the $(p,1)$ or the $(1,q)$ signature condition in the real case. To achieve our result, we will use the concept of isomaximal ideal introduced in \cite[Definition 2.3]{Kath_Olbrich_2004} and some basic facts on quadratic Lie algebras.

\begin{definition}\label{def:isomaximal}
	Let $I$ be a totally isotropic ideal in a quadratic Lie algebra $(L, \varphi)$. The ideal $I$ is called isomaximal if it is not contained in any other totally isotropic ideal.
\end{definition}

\begin{lemma}[Kath, Olbrich, 2003]\label{lem:quadratic-patterns-KO}
	Let $(L, \varphi)$ be a quadratic indecomposable Lie algebra over a field of characteristic zero, and denote by $R(L)$ the solvable radical of $L$. Then:
	\begin{enumerate}[\quad a)]
		\item $L$ has no proper simple ideals and $R(L)^\perp\subseteq R(L)$.
		\item If $I$ is an isomaximal ideal, then $I\neq L$, $I^\perp\subseteq R(L)$, and the quotient Lie algebra $\displaystyle\frac{I^\perp}{I} $ is abelian.
	\end{enumerate}
\end{lemma}

\begin{proof}
	This is Lemma 2.2 and Lemma 2.3 in \cite{Kath_Olbrich_2004}. The proofs of both lemmas hold over fields of characteristic zero.
\end{proof}

\begin{theorem}\label{thm:Kosciladoras-Witt1}
	The following assertions are equivalent:
	\begin{enumerate}[\quad a)]
		\item $(L, \psi)$ is an non-semisimple indecomposable quadratic Lie algebra such that $\psi$ has Witt index $1$.
		\item $L=\mfd(V,\varphi, \delta)$ is a $\mbK$-oscillator algebra, and $\varphi(v,v)\neq 0$ for all $v\in V$.
	\end{enumerate}
	Lie algebras that satisfy one of the previous equivalent conditions are double extensions of an abelian quadratic algebra $(V,\varphi)$ by a $\varphi$-skew semisimple automorphism $\delta$. Even more, any irreducible polynomial of the factorization of the minimum polynomial $m_\delta(x)$ is of the form $\pi(x)=x^{2n}-a_{n-1}x^{2(n-1)}-\dots-a_1x^2-a_0\in \mbK[x]$, $n\geq 1$.
\end{theorem}

\begin{proof}
	Assume condition a) and note that the maximal dimension of any totally isotropic subspace is one. From Lemma~\ref{lem:quadratic-patterns-KO}, $R(L)^\perp\subseteq R(L)$, so  $R(L)^\perp$ is a totally isotropic ideal, and therefore $d=\dim R(L)^\perp\leq 1$. Since $d=\dim L-\dim R(L)$ is just the dimension of any Levi factor of $L$, $d\geq 3$ if $d\neq 0$. Thus, $d=0$ and $L=R(L)$ is a solvable Lie algebra. This implies $L^2\neq L$, and then $Z(L)=(L^2)^\perp$ is a nonzero ideal. From Remark~\ref{rmk:decomp-indecomp}, $0\neq Z(L)\subset L^2=Z(L)^\perp$ by indecomposibility; therefore, $Z(L)$ is a totally isotropic ideal. So, $Z(L)=\mbK\cdot z$ is a minimal  ideal, $\psi(z,z)=0$, and $(Z(L))^\perp=L^2$ is a maximal ideal of codimension one. Then $L=\mbK\cdot x \oplus L^2$, and from the non-degeneracy of $\psi$, we can assume without loss of generality $\varphi(x,z)=1$, $\varphi(x,x)=0$, i.e., $\langle x,z\rangle$ is a hyperbolic plane. From \cite[Lemma 2.7]{Favre_Santharoubane_1987}, we get that $L$ is isometrically isomorphic to the double extension of the quadratic algebra $(V=\frac{L^2}{Z(L)}, \varphi)$, $\varphi(x+Z(L), y+Z(L)):=\psi (x,y)$. But the centre is an isomaximal ideal; thus, $(V=\frac{L^2}{Z(L)}, \varphi)$ is a quadratic abelian Lie algebra following b) in Lemma~\ref{lem:quadratic-patterns-KO}. Since the Witt index of $\psi$ is one, from the decomposition $L=\langle x,z\rangle\oplus\langle x,z\rangle^\perp$ and $L^2=\mbK\cdot z\oplus \langle x,z\rangle^\perp$, it is easy to check that $\varphi$ has no isotropic vectors. Note that, as $L\cong \mfd(V,\varphi, \delta)$ and $Z(L)$ is one dimensional, $\delta$ is an automorphism according to Lemma~\ref{lem:local-general-oscillator}. Finally, as $L^2$ is the only maximal ideal of $L$, if $I$ is any proper ideal, $I^\perp \subseteq L^2$ and therefore $Z(L)\subseteq (I^\perp)^\perp=I$. This implies that $L$  is indecomposable. The final assertion on $\delta$ semisimple and irreducible factors of $m_\delta(x)$ follows from the Corollary~\ref{cor:semisimplicidad}.
\end{proof}

\begin{remark}
	The class of $\mbK$-oscillator algebras described in Theorem~\ref{thm:Kosciladoras-Witt1} is broad. According to Lemma~\ref{lem:local-general-oscillator}, the quadratic dimension of any algebra in this class is $2$. Moreover, Theorem~\ref{thm:isomorphism-isometric-iso} provides criteria for isomorphism (isometric isomorphism) between two algebras of this class. In the real case, the algebraic structure of oscillator algebras provides information on the geometry of oscillator groups \cite[Theorem 5.1]{Baum_Kath_2003}.
\end{remark}

\begin{remark}
	Any indecomposable quadratic Lie algebra with trivial solvable radical is simple. For any simple Lie algebra $S$, the Killing form $\kappa(x,y)=\Tr(\ad x \ad y)$ is invariant and, from Cartan's Criteria,  $\kappa$ is non-degenerate. For the special simple Lie algebra $\mfsl(n, \mbK)$ of zero-trace matrices, the bilinear form $b(A,B)=\frac{1}{2}\Tr AB$ allows the recovery of the Killing form as $\kappa=8b$. Over the reals, $(\mfsl(2, \mbR), \lambda \kappa)$ with $\lambda>0$ are the unique simple Lorentzian algebras. All of them are isometrically isomorphic to $(\mfsl(2, \mbR), \kappa)$.
\end{remark}

\begin{corollary}[Medina, 1985]\label{lorentzian-algebras}
	The $3$-dimensional special linear algebra is the unique quadratic non-solvable indecomposable real Lorentzian algebra, and the Killing form is the unique, up to positive scalars, invariant and non-degenerate Lorentzian form. The solvable indecomposable quadratic Lorentzian algebras are the $(2n+2)$-dimensional algebras $(\mfd_{2n+2}(\lambda),\varphi_\delta)$ described in Example~\ref{ex:real-oscillator} with $\mfd_{2n+2}(\lambda)=\spa_\mbR\langle \delta, x_i,y_i, \delta^*\rangle$ and $n\geq 1$, $\lambda = (1,\lambda_2, \dots, \lambda_n)$, where $0<\lambda_i\leq \lambda_{i+1}$.
\end{corollary}

\begin{proof}
	For the solvable case, apply Theorem~\ref{thm:isomorphism-isometric-iso} and Theorem~\ref{thm:Kosciladoras-Witt1} to prove that any $\mbR$-oscillator determined by $\lambda=(\lambda_1, \lambda_2, \dots, \lambda_n)$ is isometrically isomorphic to that described through $\displaystyle\frac{1}{\lambda_1}\lambda$.
\end{proof}

\begin{corollary}
	The invariant and non-degenerate bilinear forms of the Lorentzian algebra $(\mfd_{2n+2}(\lambda),\varphi_\delta)$ where $\lambda = (1,\lambda_2, \dots, \lambda_n)$ and $0<\lambda_i\leq \lambda_{i+1}$ are $\varphi_{t,s}$ with $(t,s)\in \mbR^2$ and $s\neq 0$ defined by:
	\[\varphi_{t,s}(\delta,\delta)=t,\; \varphi_{t,s}(\delta,\delta^*)=s,\; \varphi_{t,s}(x_i,x_j)=\varphi_{t,s}(y_i,y_j)=\delta_{ij}s,\]
	and $\varphi_{t,s}(\delta,a)=\varphi_{t,s}(\delta^*,a)=\varphi_{t,s}(\delta,a)=0$ for $a=x_i,y_j$. Then, $\varphi_\delta=\varphi_{0,1}$. Furthermore, $(\mfd_{2n+2}(\lambda), \varphi_{t,s})$ is isometrically isomorphic to $(\mfd_{2n+2}(\lambda'), \varphi_{t',s'})$ if and only if $\lambda=\lambda'$ and $ss'>0$. In particular, any $\mbR$-oscillator algebra admits two non-isometrically isomorphic quadratic structures: the one given by the invariant form $\varphi_{0,1}$  of signature $(2n-1,1)$ and the one given by the invariant form  $\varphi_{0,-1}=-\varphi_{0,1}$ of signature $(1,2n-1)$.
\end{corollary}

\begin{proof}
	The bilinear forms come from the fact that the quadratic dimension of these algebras is $2$. Now,
	for the isomorphism assertion, let $\mfd_{2n+2}\left( \lambda \right) =\spa_{\mbR}\langle \delta, x_{i},
		y_{i}, \delta^*\rangle$ and $\mfd_{2n+2}\left( \lambda' \right) =\spa_{\mbR}\langle\delta',x_{i}',
		y_{i}',\delta'^*\rangle$. By Theorem \ref{thm:isomorphism-isometric-iso} for these algebras to
	be isomorphic there must be $0\neq\mu\in\mbR$ and an isomorphism $f\colon\spa_{\mbR}\langle
		x_{i},y_{i}\rangle\to\spa_{\mbR} \langle x_{i}',y_{i}'\rangle$ such that $\delta=\mu f^{-1}\delta' f$.
	The characteristic polynomial for the map on the rigth side is equal to
	$q(x)=\left( x^{2}+\mu^{2} \right) \left( x^{2}+\left( \mu\lambda_{2}' \right)^{2}  \right)\cdots
		\left( x^{2}+\left( \mu\lambda_{n}' \right) ^{2} \right) $ with $\mu\leq \mu\lambda_{2}'\leq\cdots
		\leq\mu\lambda_{n}'$. The characteristic polynomial for $\delta$ is $p(x)=\left( x^{2}+1 \right)
		\left( x^{2}+\lambda_{2}^{2} \right) \cdots\left( x^{2}+\lambda_{n}^{2} \right) $ with $1\leq\lambda_{2}
		\leq\cdots\leq\lambda_{n}$. Since they must be equal, $1=\mu$, $\lambda_{i}=\mu\lambda_{i}'$, this
	happen if and only if $1=\mu$ and $\lambda_{i}=\lambda_{i}'$, that is, $\lambda=\lambda'$.
	\[
		\varphi_{t,s}\left( \delta,\delta^* \right) =s=\lambda\mu s'=\varphi_{t',s'}\left( F\left( \delta \right)
		,F\left( \delta^* \right) \right)
		.\]
	And thus
	\[
		\frac{s}{s'}\varphi\left( x,y \right) =\varphi\left( f\left( x \right) ,f\left( y \right)  \right) ,\quad
		x,y\in\mfd_{2n+2}\left( \lambda \right)
		.\]
	For the isometric part, assume that $F\colon\left(\mfd_{2n+2}\left( \lambda \right),\varphi_{t,s}\right)
		\to\left(\mfd_{2n+1}\left( \lambda\right),\varphi_{t',s'}\right) $ is an isometric isomorphism where
	$F$ is as in Theorem \ref{thm:isomorphism-isometric-iso}. Write
	$f(x_{1}) =\sum_{k=1} ^{n}a_{k}x_{k}+b_{k}x_{k}$ some $a_{k}$ or $b_{k}$ is not zero since $f$ is an
	isomorphism. Then
	\[
		\frac{s}{s'}=\frac{s}{s'}\varphi\left( x_{1},x_{1} \right) =\varphi\left( f\left( x_{1} \right) ,
		f\left( x_{1} \right) \right) =\sum_{k=1}^{n}a_{k}^{2}+b_{k}^{2}>0
		.\]
	Or, equivalently, $ss'>0$.

	Conversely, let $s=\pm 1$, $t=0$ and $s'$ such $ss'>0$. Define the map $F\colon \left( \mfd_{2n+1}
		\left( \lambda \right) ,\varphi_{0,s}\right) \to\left( \mfd_{2n+1}\left( \lambda \right) ,\varphi_{t',s'}
		\right) $ by
	\[
		F(\delta)=\delta+\nu\delta^*,\quad F(x)= \sqrt{\frac{s}{s'}} x,\quad F\left( \delta^* \right) =\frac{s}{s'}\delta^*,
	\]
	where $\nu$ satisfies $t'+2\nu s'=0$. By Theorem \ref{thm:isomorphism-isometric-iso} this map is an
	isomorphism and straightforward computations check that it is also an isometry. So $(\mfd_{2n+2}(\lambda), \varphi_{t',s'>0})\cong (\mfd_{2n+2}(\lambda), \varphi_{0,1 })$ and $(\mfd_{2n+2}(\lambda), \varphi_{t',s'<0})\cong (\mfd_{2n+2}(\lambda), \varphi_{0,-1})$.
\end{proof}


\section*{Acknowledgements}
This research has been partially funded by grant Fortalece 2023/03 of ``Comunidad Autónoma de La Rioja'' and by grant MTM2017-83506-C2-1-P of ``Ministerio de Economía, Industria y Competitividad, Gobierno de España'' (Spain) from 2017 to 2022 and by grant PID2021-123461NB-C21, funded by MCIN/AEI/10.13039/501100011033 and by ``ERDF: A way of making Europe'' starting from 2023. The third author has been also supported until 2022 by a predoctoral research contract FPI-2018 of ``Universidad de La Rioja''.


\bibliographystyle{apalike}
\bibliography{bibliography.bib}
\vfill

\end{document}